\documentclass[12pt,a4paper]{amsart}
\usepackage{amsfonts}
\usepackage{amssymb}
\usepackage{amsmath}
\usepackage{amsthm}
\usepackage{cite}
\usepackage{enumitem}
\usepackage{tikz}
\usepackage{subfigure}
\usepackage[normalem]{ulem}
\usepackage{cancel}
\usepackage[pagebackref=false]{hyperref}
\usepackage{float}
\usepackage{graphicx}

\theoremstyle{plain}
\newtheorem{thm}{Theorem}

\newtheorem{lem}[thm]{Lemma}
\newtheorem{prop}[thm]{Proposition} 

\newtheorem{remark}[thm]{Remark}

\usetikzlibrary{decorations.markings}
\tikzset{->-/.style={decoration={
  markings,
  mark=at position #1 with {\arrow{>}}},postaction={decorate}}}
	\tikzset{-<-/.style={decoration={ 
  markings, 
  mark=at position #1 with {\arrow{<}}},postaction={decorate}}}

\newcommand{\RR}{\mathbb{R}}
\newcommand{\CC}{\mathbb{C}}

\newcommand{\abs}[1]{\left\vert #1 \right\vert}
\newcommand{\norm}[1]{\left\Vert #1 \right\Vert}
\newcommand{\paren}[1]{\left( #1 \right)}

\newcommand{\sign}{\operatorname{sign}}

\newcommand{\DEL}[1]{}

\providecommand{\N}{\mathbb{N}}  
\providecommand{\R}{\mathbb{R}}

\providecommand{\Z}{\mathbb{Z}}
\providecommand{\C}{\mathbb{C}}
\providecommand{\T}{\mathbb{T}}

\providecommand{\eps}{\varepsilon}
  
\providecommand{\ov}{\overline}

\DeclareMathOperator{\Id}{Id}

\DeclareMathOperator{\Real}{Re}

\DeclareMathOperator{\spann}{span}
\DeclareMathOperator{\rg}{Rg}
\DeclareMathOperator{\Kern}{Ker}

\providecommand{\N}{\mathbb{N}}
\providecommand{\R}{\mathbb{R}}
\providecommand{\Z}{\mathbb{Z}}
\providecommand{\C}{\mathbb{C}}
\providecommand{\eps}{\varepsilon}
\providecommand{\ov}{\overline}

\DeclareMathOperator{\sech}{sech}
\renewcommand{\qed}{\hfill $\Box$}
\def\i{\mathrm {i}}
\def\e{\mathrm {e}}
\renewcommand{\Re}{\operatorname{Re}}
\renewcommand{\Im}{\operatorname{Im}}

\def\XXint#1#2#3{{\setbox0=\hbox{$#1{#2#3}{\int}$} 
     \vcenter{\hbox{$#2#3$}}\kern-.5\wd0}}

\setitemize{itemsep=+2pt}
\setenumerate{itemsep=+2pt}

\begin{document}

\title[The Lugiato-Lefever equation with nonlinear damping]{The Lugiato-Lefever equation with nonlinear damping caused by two photon absorption}
 
\author{Janina G\"artner, Rainer Mandel, Wolfgang Reichel}

\address{~\hfill\break 
J. G\"artner \hfill\break
Karlsruhe Institute of Technology \hfill\break
Institute for Analysis \hfill\break
Englerstra{\ss}e 2 \hfill\break
D-76131 Karlsruhe, Germany}
\email{Janina.Gaertner@kit.edu}
\address{~\hfill\break 
R. Mandel \hfill\break
Karlsruhe Institute of Technology \hfill\break
Institute for Analysis \hfill\break
Englerstra{\ss}e 2 \hfill\break
D-76131 Karlsruhe, Germany}
\email{Rainer.Mandel@kit.edu}
\address{~\hfill\break 
W. Reichel \hfill\break
Karlsruhe Institute of Technology \hfill\break
Institute for Analysis \hfill\break
Englerstra{\ss}e 2 \hfill\break
D-76131 Karlsruhe, Germany}
\email{Wolfgang.Reichel@kit.edu}

\date{}


\begin{abstract}
  In this paper we investigate the effect of nonlinear damping on the Lugiato-Lefever equation 
  $$
  \i \partial_t a = -(\i-\zeta) a - da_{xx} -(1+\i\kappa)|a|^2a +\i f
  $$ 
  on the torus or the real line. For the case of the torus it is shown that for small nonlinear damping
  $\kappa>0$ stationary spatially  periodic solutions exist on branches that bifurcate from constant
  solutions whereas all nonconstant solutions disappear when the damping parameter $\kappa$ exceeds a
  critical value. These results apply both for normal ($d<0$) and anomalous ($d>0$) dispersion. For the
  case of the real line we show  by the Implicit Function Theorem  that for small nonlinear damping $\kappa>0$
  and large detuning $\zeta\gg 1$ and large forcing $f\gg 1$ strongly localized, bright solitary stationary
  solutions exists in the case of anomalous dispersion $d>0$. These results are achieved by using techniques
  from bifurcation  and continuation theory and by proving a convergence result for solutions of the time-dependent Lugiato-Lefever equation.
\end{abstract}

\maketitle


\section{Introduction}


\medskip

The Lugiato-Lefever equation
\begin{equation} \label{eq:LLE_timedependent_standard} 
   \i \partial_t a = -(\i-\zeta) a - da_{xx} -|a|^2a +\i f 
\end{equation}
was proposed in 1987 by Lugiato and Lefever~\cite{Lugiato_Lefever1987}
as an approximative model for the electric field inside an optical cavity excited by a laser pump of strength $f$.  
Since then many authors have derived~\eqref{eq:LLE_timedependent_standard} as a model, e.g., for the field
$a(x,t)=\sum_{k\in\Z} \hat a_k(t) \e^{\i kx}$ inside a continous wave(cw)-pumped ring resonator, cf. \cite{Chembo_Menyuk, chembo_2010,
herr_2012}. Here $\hat a_k(t)$ denotes the complex amplitude of the $k$-th excited mode in the ring
resonator. The cw-laser frequency has a detuning offset $\zeta$ relative to the primarily excited $0$-mode of
the ring resonator, and the second-order linear dispersion coefficient $d$ of the ring resonator may be
normal ($d<0$) or anomalous ($d>0$). Nonlinear interaction of the strongly enhanced field due to the Kerr
effect in the microresonator eventually leads to modulation instability. Consequently, a cascaded transfer of
power from the primarily excited mode to a multitude of neighbouring modes takes place. A resulting stable
stationary pattern of spectrally equidistant excited modes is called a frequency comb. Spectrally broad
octave spanning frequency combs have turned out to be extremely attractive sources for a variety of
applications including time and frequency metrology \cite{Diddams1999,Udem2002}, high-speed optical data
communications \cite{Pfeifle2015,Pfeifle_Koos3,Marin-Palomo2016}, and ultrafast optical ranging
\cite{Suh2017,Trocha2017}.

\medskip

Recently, semiconductors exhibiting two-photon-absorption (TPA) at telecommunication wavelengths such as silicon have been considered as waveguide materials for microresonators. TPA causes an electron from the valence band to be excited to the conduction band. There, free-carrier absorption (FCA) of additional photons leads to a further excitement to other states within the conduction band. While these nonlinear losses hinder the generation of frequency combs in microresonators, at the same time comb formation benefits from a higher Kerr nonlinearity that comes along with TPA. Furthermore, especially silicon is highly relevant from a practical point of view, since it is an established material used for photonic integrated circuits.

\medskip

We are not aware of mathematically rigorous studies on the Lugiato-Lefever equation with TPA or FCA. 
In this paper we want to start the analysis of the effect of TPA on the formation of frequency combs. For
mathematical reasons the effect of FCA will be neglected in this paper, since the full model is currently out
of reach for our analysis. TPA modifies the Kerr effect by adding an imaginary component $\i \kappa$,
$\kappa>0$ to the coefficient of the cubic nonlinear susceptibility. Following \cite{Hansson2016, Lau:15} the
model equation \eqref{eq:LLE_timedependent_standard} is therefore modified as follows
\begin{equation} \label{eq:LLE_timedependent} 
   \i \partial_t a = -(\i-\zeta) a - da_{xx} -(1+\i\kappa)|a|^2a +\i f.
\end{equation}
Since FCA will not be considered we have set the free carried density to $0$ so that the ODE for the free carrier density, 
which is coupled to \eqref{eq:LLE_timedependent}, cf. \cite{Hansson2016, Lau:15}, is not present.
Stationary solutions of \eqref{eq:LLE_timedependent} satisfy
\begin{equation} \label{eq:LL_complex} 
   -da''-(\i-\zeta)a -(1+\i\kappa)|a|^2a +\i f=0, \qquad a(\cdot)=a(\cdot+2\pi) 
\end{equation}
where the spatial period given by the circular nature of resonators is normalized to $2\pi$. Due to
the nonlinear damping effect of TPA in addition to the linear damping, TPA is unfavorable for comb formation.
However, in this paper we prove the converse: Kerr comb formation in silicon based microresonators is  still
possible if  the TPA coefficient $\kappa$ is sufficiently small. For large
$\kappa$ above a certain threshold, for which we provide lower bounds, Kerr comb formation is prohibited. Our
results apply both for normal and anomalous dispersion. Since soliton-like stationary solutions
of~\eqref{eq:LLE_timedependent} are of utmost importance in applications, we also consider the formation of
bright solitary combs for anomalous dispersion in the presence of small $\kappa$.

\medskip

Before describing our results for \eqref{eq:LLE_timedependent} and \eqref{eq:LL_complex} in
more detail, we first present the mathematical results which deal with the special case
$\kappa=0$ of purely linear damping. 
One important fact about~\eqref{eq:LL_complex} for $\kappa=0$ and any fixed $f\neq 0$ is that there is a
uniquely determined curve parameterized by $\zeta$ consisting of constant solutions, see for instance
Lemma~2.1~(a)~\cite{MaRe_aprioribounds} for an explicit parametrization. With $\zeta$ as a bifurcation
parameter bifurcation theory is a convenient tool for proving the existence of nonconstant solutions. A number of existence results
for~\eqref{eq:LL_complex}  with $\kappa=0$  were found using bifurcation results for dynamical systems via
the spatial dynamics approach
\cite{Godey_BifurcationAnalysis,Godey_et_al2014,DelHara_periodic,Parra-Rivas2014,Parra-Rivas2016,ParRivGomKnob_Bifurcation}.
Here the requirement of $2\pi$-periodicity is dropped and one is interested in nonconstant solutions of the
four-dimensional (real) dynamical system that corresponds to the second order ODE from~\eqref{eq:LL_complex}
for the complex-valued function $a$. A detailed analysis of the normal forms of this system around the
constant equilibria reveals which types of solutions exist in a neighbourhood.  
In \cite{Godey_BifurcationAnalysis} (Theorem~2.1--2.6) periodic, quasiperiodic and homoclinic orbits
were proved to exist near the curve of constant solutions both in the case of normal dispersion $(d<0)$  and 
anomalous dispersion $(d>0)$.  Since solutions corresponding to these orbits necessarily resemble constant
functions on $[0,2\pi]$, soliton-like solutions with a strong spatial profile can not be analytically
described by local bifurcation methods. Therefore, in order to see interesting spatial profiles, local
bifurcations have to be continued, e.g., by numerical methods, cf.
\cite{MaRe_aprioribounds,Parra-Rivas2014,Parra-Rivas2016,ParRivGomKnob_Bifurcation}, far away from the curve
of constant equilibria.
          
\medskip               
              
Proving local bifurcations of exactly $2\pi$-periodic solutions 
requires a different approach. A first local bifurcation bifurcation result from a specific constant
solution was proved in~\cite{Miyaji_Ohnishi_Tsutsumi2010} (Theorem~3.1). This
study was extended in \cite{MaRe_aprioribounds} using local and global bifurcation results due to
Crandall-Rabinowitz and Krasnoselski-Rabinowitz. All (finitely
many) bifurcation points on the curve of constant solutions were identified and the bifurcating solutions
were shown to lie on bounded solution continua that return to another bifurcation point. Some of these
continua even undergo  period-doubling, period-tripling, etc. secondary bifurcations as was shown in
Section~4 in~\cite{Mandel_secondary}.
The theoretical results  from \cite{MaRe_aprioribounds,Mandel_secondary} were accompanied by numerically
computed bifurcation diagrams indicating that the most localized and thus soliton-like solutions can be found
at those turning points of the branches that are the farthest away from the curve of trivial solutions. We remark that a two-dimensional version of the Lugiato-Lefever equation posed on the unit disk was recently
discussed in~\cite{MiyTsu_Steadystate}.

\medskip

Finally, still in the case $\kappa=0$ we mention some results about the time-dependent
equation~\eqref{eq:LLE_timedependent_standard}.
In~\cite{jami:14} it was proved that the initial value problem is globally well-posed in 
$a\in C(\R_+,H^4(\T))\cap C^1(\R_+,H^2(\T))\cap C^2(\R_+,L^2(\T))$ for initial data
in $H^4(\T)$. Here, $\T$ is the one-dimensional torus, i.e., the interval $[0,2\pi]$ with both ends
identified, and $\R_+=[0,\infty)$ is the temporal half-line. Additionally, it was shown that all solutions of
the initial value problem remain bounded in $L^2$ while the $H^1$-norm is proved to grow at most like
$\sqrt{t}$ as $t\to\infty$. In the corresponding model with an additional third order dispersion effect
well-posedness results and even the existence of a global attractor were proved
in~\cite{MiyTsu_existence}. Convergence results for the numerical Strang-splitting scheme can be found 
in~\cite{jami:14}. Finally, the orbital asymptotic stability of $2\pi$-periodic solutions was investigated
in~\cite{StanStef_asymptotic} (Theorem~1) with the aid
of the Gearhart-Pr\"uss-Theorem, see
also~\cite{Miyaji_Ohnishi_Tsutsumi2011,Miyaji_Ohnishi_Tsutsumi2011_Erratum}.
 Notice that the linearized operators (i.e. the generators of the
semigroup) are not selfadjoint, which makes this result particularly interesting. Using the center manifold approach, spectral stability and instability results as well as nonlinear stability with respect to co-periodic or subharmonic perturbations were obtained in \cite{DelHara_periodic}.

\medskip

Let us now describe the results of our paper. We consider \eqref{eq:LL_complex} with $f\not =0$, $\kappa\geq 0$ and $d\not =0$ fixed. Our first theorem
contains three results on the structure of solutions of \eqref{eq:LL_complex}.  Notice that for every
$\zeta\in\R$ \eqref{eq:LL_complex} has either one, two or three different constant solutions $a_0\in \C$ lying
on a smooth curve.  Theorem~\ref{thm:bifurcations} addresses the question of bifurcation from the curve of
trivial solutions.   We show that for sufficiently small $\kappa\in (0,1/\sqrt{3})$ bifurcation from the
curve of trivial solutions happens, whereas  for sufficiently large $\kappa>\kappa_\ast$ the trivial curve
has no bifurcation points at all. In case of small $\kappa$ we give sufficient conditions \eqref{integer},
\eqref{transversality} for bifurcation based on the Crandall-Rabinowitz theorem on bifurcation from simple
eigenvalues \cite{CrRab_bifurcation}. They correspond to simple kernels of the linearization around a given
point of the trivial curve and to transversality, respectively.

\medskip
 
The notion of bifurcation may depend on spaces and norms. In our context we use the following set-up. Let $\T$ be the one-dimensional torus, 
i.e., the interval $(0,2\pi]$ with end-points $0$ and $2\pi$ identified. We consider solutions $a=\Re a+\i
\Im a\in H^2(\T)$ of \eqref{eq:LL_complex}.

 \begin{thm}\label{thm:bifurcations} 
    For $f\neq 0, \kappa>0$ the following holds:
    \begin{itemize}
      \item[(i)] All constant solutions of \eqref{eq:LL_complex} form a smooth unbounded curve in $H^2(\T)\times\R$.
      \item[(ii)] A point $(\zeta,a_0)$ on the curve of constant solutions is a bifurcation point provided exactly one of the two numbers
      \begin{equation}\label{integer}
	k_{1,2} := \sqrt{\frac{2|a_0|^2-\zeta \pm \sqrt{(1-3\kappa^2)|a_0|^4-4\kappa |a_0|^2-1}}{d}}
    \end{equation}
    is in $\N$ and 
    \begin{multline} \label{transversality}
    2(3\kappa^2-|a_0|^4)(|a_0|^2-\zeta)-4\kappa|a_0|^2(3|a_0|^2-\zeta) \\
    \pm \sqrt{(1-3\kappa^2)|a_0|^4-4\kappa |a_0|^2-1}\Bigl(1+\zeta^2-|a_0|^4-4\kappa|a_0|^2+3\kappa^2\Bigr)
    \not =0
    \end{multline} 
    with ``$+$'' if $k_1\in\N$ and ``$-$'' if $k_2\in\N$.
    \item[(iii)] The curve of constant solutions does not contain bifurcation points
      provided $\kappa>\kappa_*$ where 
      \begin{align*}
        \kappa_* & := \max\left\{\kappa\in (0,\frac{1}{\sqrt{3}}):
        \frac{2\kappa+\sqrt{1+\kappa^2}}{(1-3\kappa^2)^3} (1-\kappa^2+\kappa\sqrt{1+\kappa^2})^2 \leq   f^2\right\} \mbox{ if } f^2>1, \\
        \kappa_* & :=0 \mbox{ if } f^2\leq 1.
      \end{align*}
    \end{itemize} 
  \end{thm}
  
  \begin{remark} 
	\begin{itemize}
		\item[(i)] Necessarily, we have $\kappa <\sqrt{3}$ in case (ii) since otherwise the values $k_{1,2}$
    in~\eqref{integer} will not be real. Moreover, in case (ii) we may apply Rabinowitz' global bifurcation
    theorem from \cite{Rab_some_global}. It says not only that $(\zeta,a_0)$ is a bifurcation point, but
    that there is a global branch of non-trivial solutions that either returns to the trivial
    branch at some other bifurcation point or becomes unbounded in the $\zeta$-direction or in the
    $H^2(\T)$-direction.
		\item[(ii)] Notice that by strict monotonicity, the value $\kappa_*$ is the uniquely determined solution of 
	\begin{align}\label{eq:zero_1}
     \frac{2\kappa+\sqrt{1+\kappa^2}}{(1-3\kappa^2)^3} (1-\kappa^2+\kappa\sqrt{1+\kappa^2})^2 = f^2,
   \end{align}
	cf. Figure~\ref{curve}.
	\item[(iii)] For $|f|\searrow 1$, we have $\kappa_*\to 0$. This is consistent with \cite{MaRe_aprioribounds}, where for $\kappa=0$ it was shown that no bifurcations occur in the case $|f|\leq 1$. 
	\item[(iv)] By running \texttt{pde2path} for increasing values of $\kappa>0$ we can determine numerically when bifurcations cease to exist. The values for $\kappa_\star$ from Theorem~\ref{thm:bifurcations} and these numerically determined values from \texttt{pde2path} are very similar, cf Table\ref{crit_kappa}.
	\end{itemize}
  \end{remark}
   
\begin{table}[ht]
	\begin{minipage}[t]{0.44\textwidth}
		\begin{center} 
			\begin{figure}[H]\vspace{-2.5cm}
				\begin{tikzpicture}[xscale=10,yscale=0.02]
					\draw[->] (-0.05,0) -- (0.6,0) node[right] {$\kappa$};
					\draw[->] (0,-25) -- (0,120);
					\draw[-](0.55,-10)--(0.55,10) node[below,yshift=-0.3cm](label) {\small{$1/\sqrt{3}$}};
					\draw[scale=1,domain=-0.:0.502,smooth,variable=\kappa,blue] plot ({\kappa},{2*\kappa+sqrt(1+\kappa*\kappa)/((1-3*\kappa*\kappa)*(1-3*\kappa*\kappa)*(1-3*\kappa*\kappa))*(1-\kappa*\kappa+\kappa*sqrt(1+\kappa*\kappa))*(1-\kappa*\kappa+\kappa*sqrt(1+\kappa*\kappa))});
					\draw[dashed] (0.47,50) -- (-0.005,50) node[left] {$f^2$};
					\draw[dashed] (0.47,50) -- (0.47,-5) node[below] {$\kappa_\star$};
				\end{tikzpicture}
			\caption{Illustration of~\eqref{eq:zero_1}.}
			\label{curve}
			\end{figure}
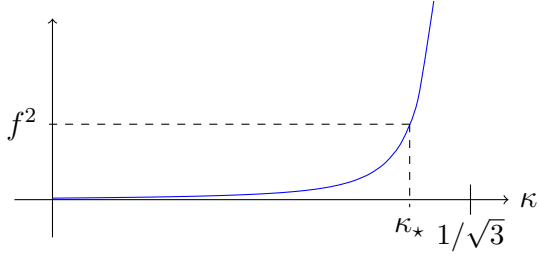
		\end{center}
	\end{minipage}
	\begin{minipage}[t]{0.55\textwidth}
		\begin{center}
			\begin{tabular}{|c|c|c|}\hline
				$f$ & $\kappa_\star$ & $\kappa_\star^{num}$\\ \hline
				$1.1$ & $0.045$ & $0.042$\\ 
				$1.6$ & $0.185$ & $0.185$\\ 
				$2$ & $0.248$ & $0.245$ \\
				$4$ & $0.380$ & $0.378$\\ 
				$10$ & $0.474$ & $0.473$ \\
				$20$ & $0.513$ & $0.513$\\ \hline
			\end{tabular}
		\end{center}
	\caption{$\kappa_\star$ from Theorem~\ref{thm:bifurcations} and numerical values from \texttt{pde2path}.}
	\label{crit_kappa}
	\end{minipage}
\end{table}

  Theorem~\ref{thm:bifurcations} provides nontrivial solutions via bifurcation theory for $\kappa\in
  (0,\kappa_*)$, i.e., the bifurcating branches described in~\cite{MaRe_aprioribounds} for $\kappa=0$ persist
  for small $\kappa>0$. The natural question, what happens to the bifurcating branches when $\kappa$ gets
  larger, is also answered in part (iii) of the theorem: bifurcation points disappear at latest when $\kappa$
  exceeds $\kappa_*$. In Figure~\ref{bif_dig} the vanishing of bifurcation points and nontrivial solutions
  for increasing $\kappa$ is illustrated. Black curves indicate the line of trivial solutions, colored curves
  show bifurcation branches. With increasing nonlinear damping, more and more bifurcation branches vanish,
  until all have disappeared when $\kappa$ exceeds the value $0.185$.
    
\begin{figure}[H]
\centering
\begin{minipage}[t]{0.31\columnwidth}
\includegraphics[width=\textwidth]{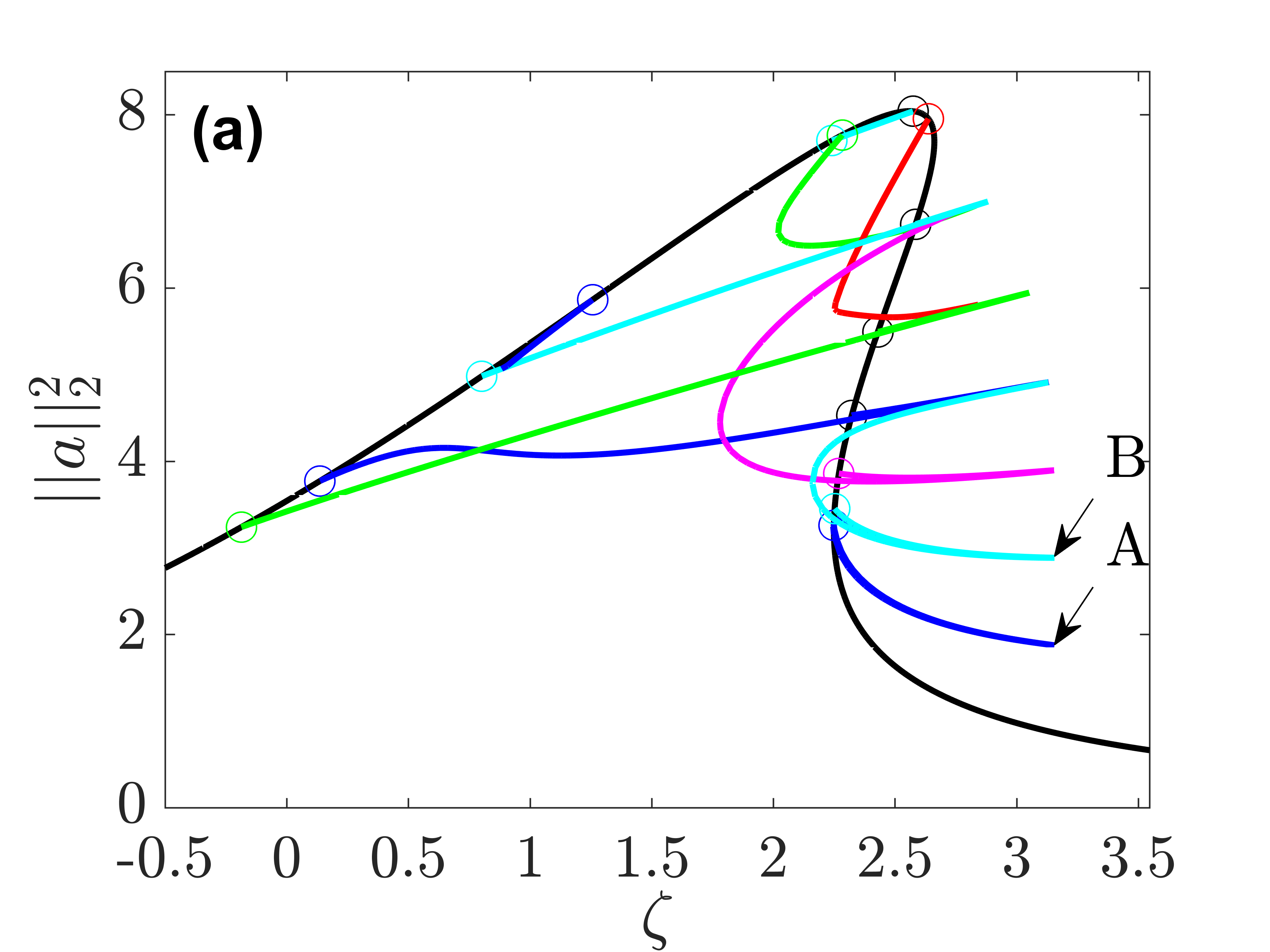}
\includegraphics[width=\textwidth]{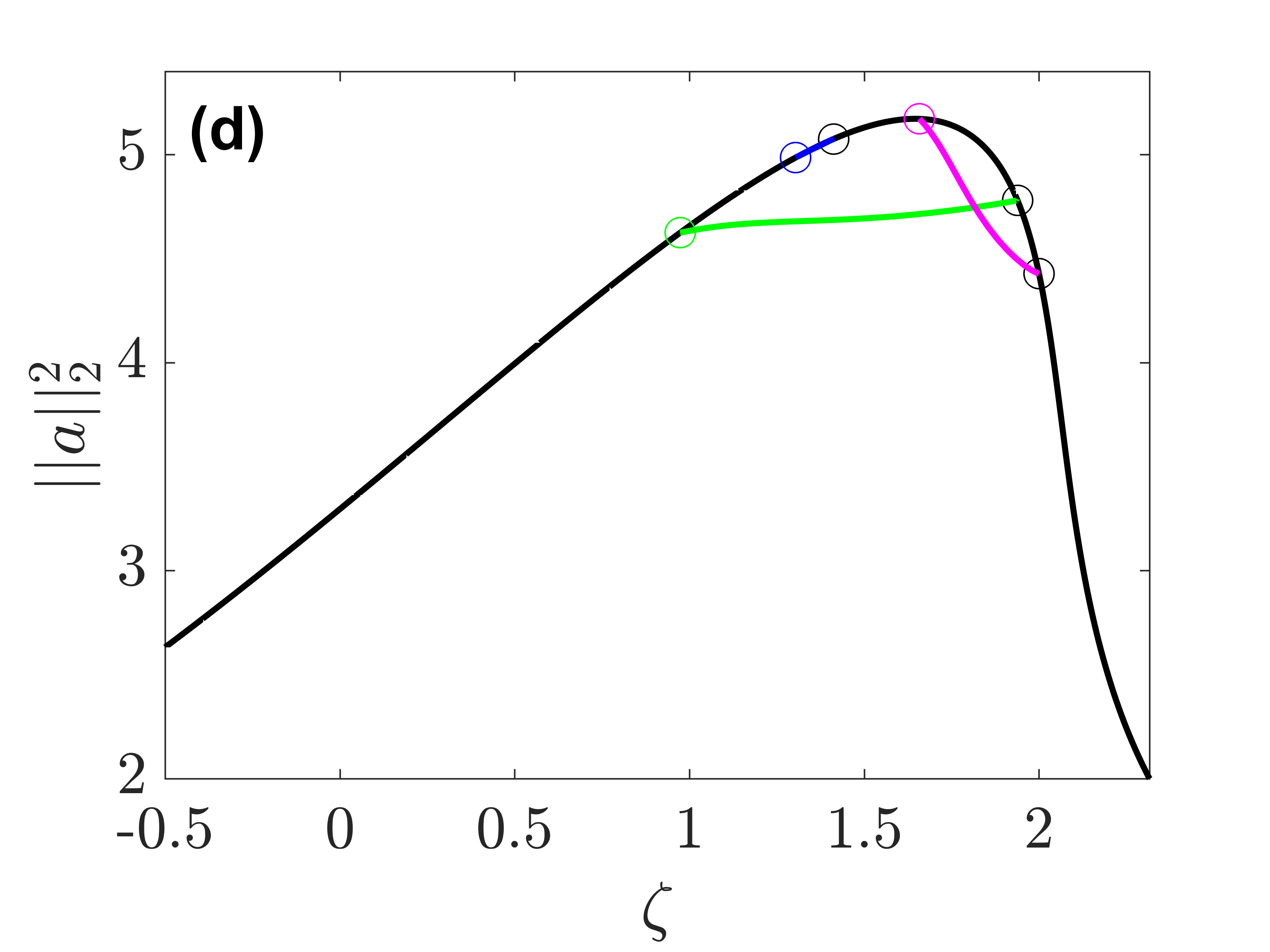}
\end{minipage} 
\begin{minipage}[t]{0.31\columnwidth}
\includegraphics[width=\textwidth]{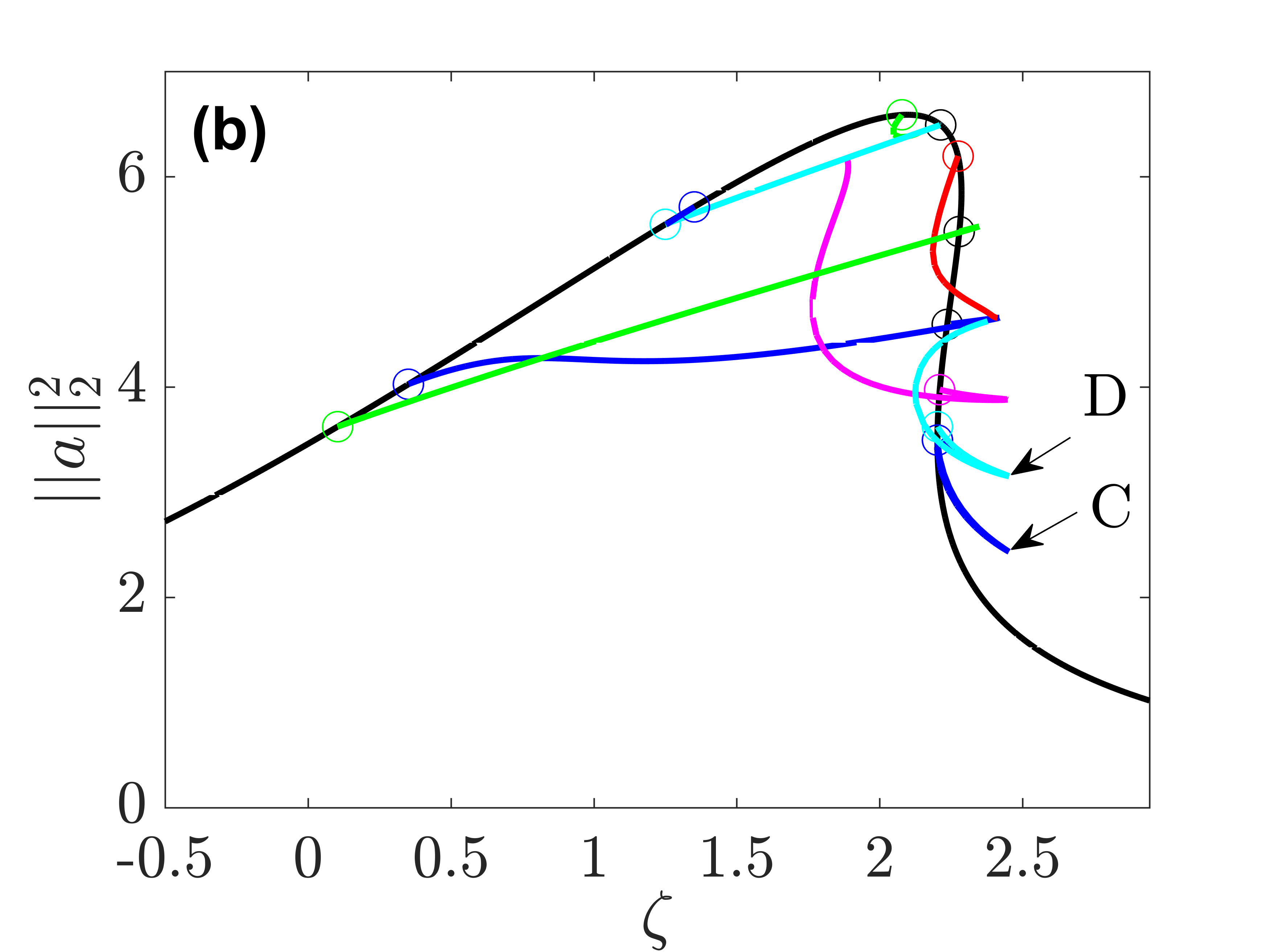}
\begin{tikzpicture}[overlay]
\node at (2.53,-1.47)
       {\includegraphics[width=\columnwidth]{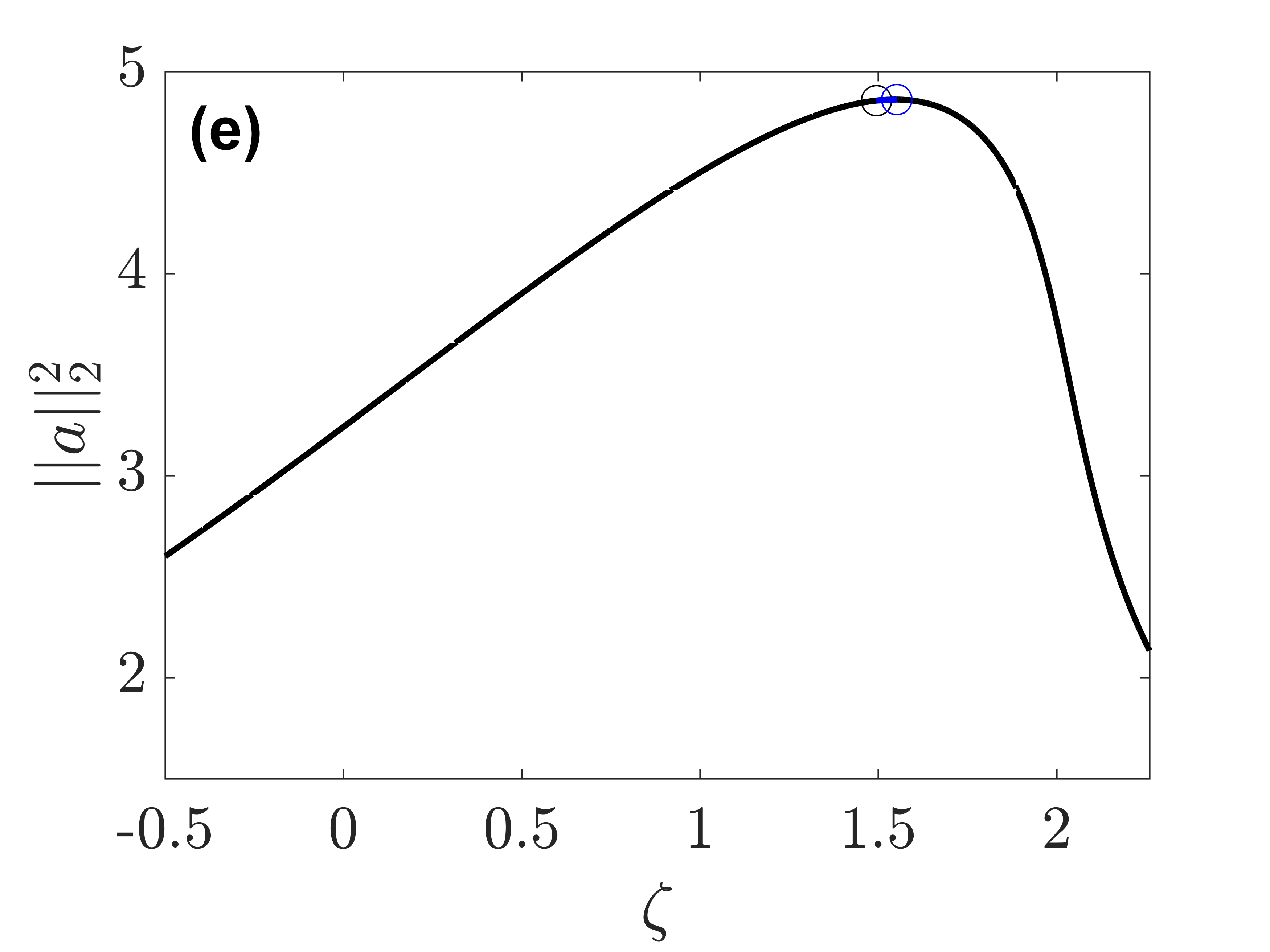}};
\node  at (3.14,-1.9) {\includegraphics[width=0.4\columnwidth]{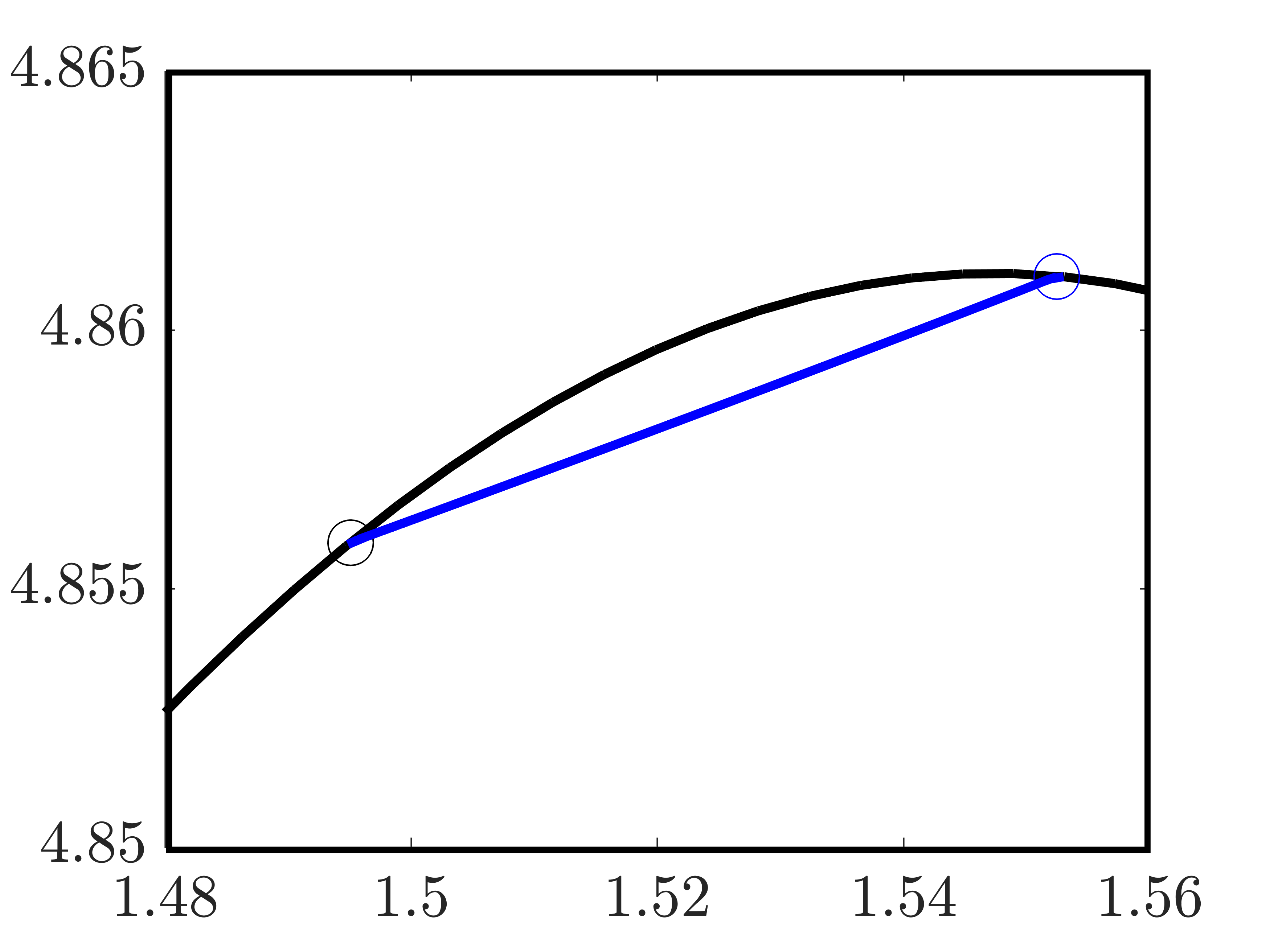}};
\draw [draw=black,thick] (3.84,0.16) rectangle (3.32,-0.1);
\draw (3.33,-0.1) -- (2.37,-1.25);
\draw (3.85,-0.1) -- (3.97,-1.25);
\end{tikzpicture}
\end{minipage}
\begin{minipage}[t]{0.31\columnwidth}
\includegraphics[width=\textwidth]{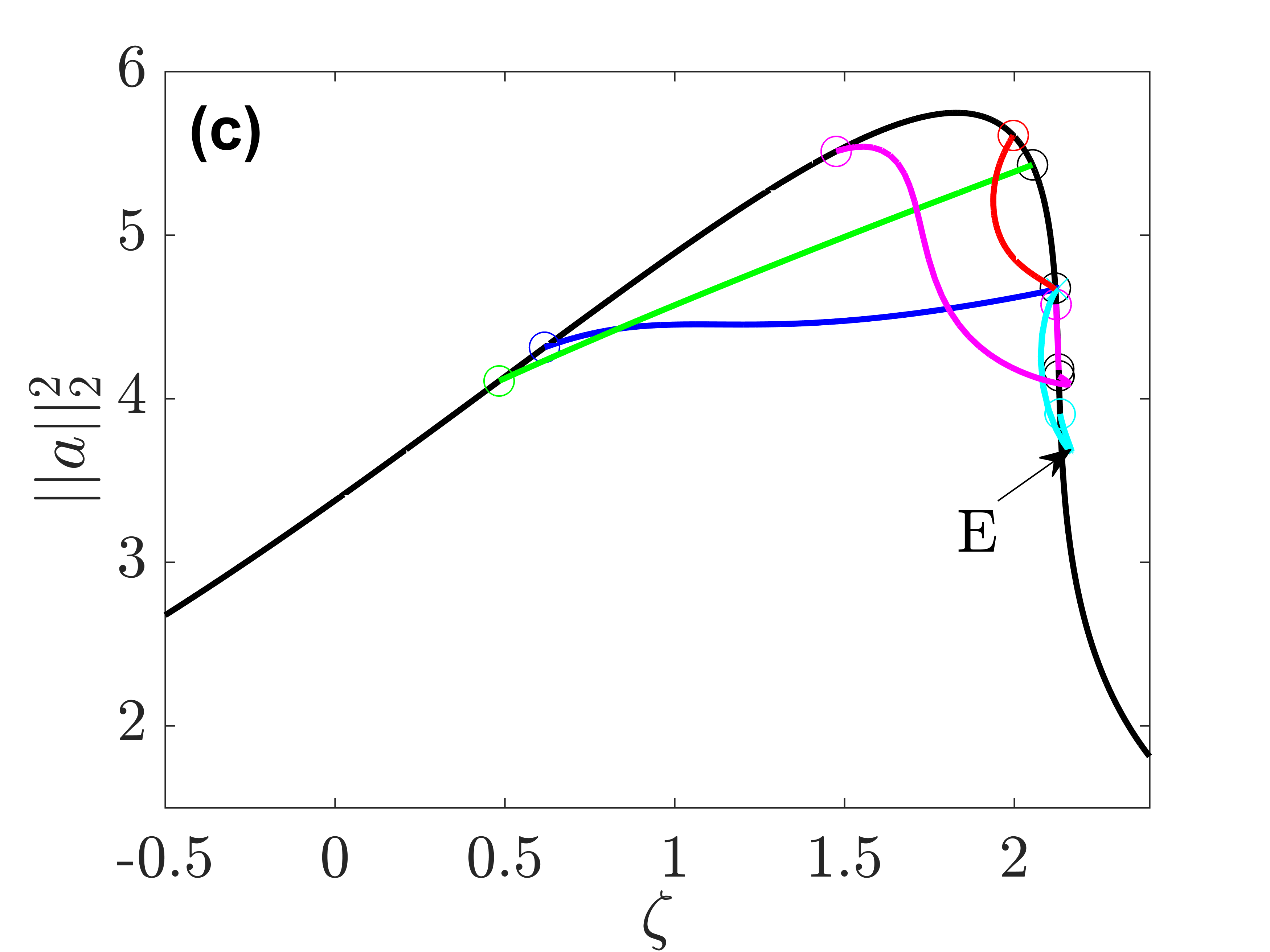}
\includegraphics[width=\textwidth]{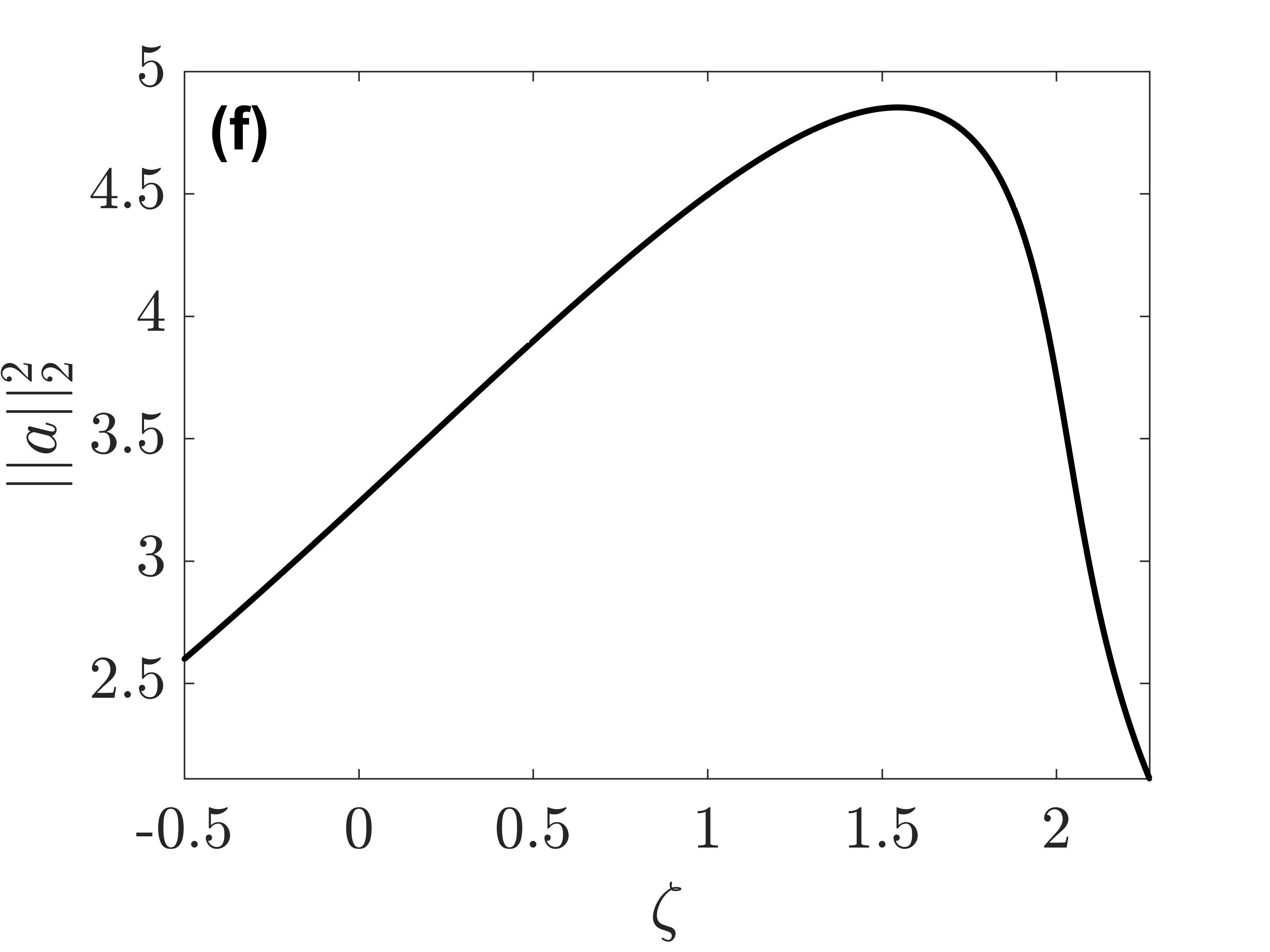}
\end{minipage}
\caption{Bifurcation diagrams for $d=0.1$, $f=1.6$. Subfigure (a) corresponds to $\kappa=0$ , (b) to
$\kappa=0.05$, (c) to $\kappa=0.1$, (d) to $\kappa=0.15$, (e) to $\kappa=0.185$, (f) to $\kappa=0.186$.
Solutions at turning points A, B in (a), C, D in (b) and E in (c) are shown in
Figure~\ref{solitons}.}\label{bif_dig}
\end{figure}
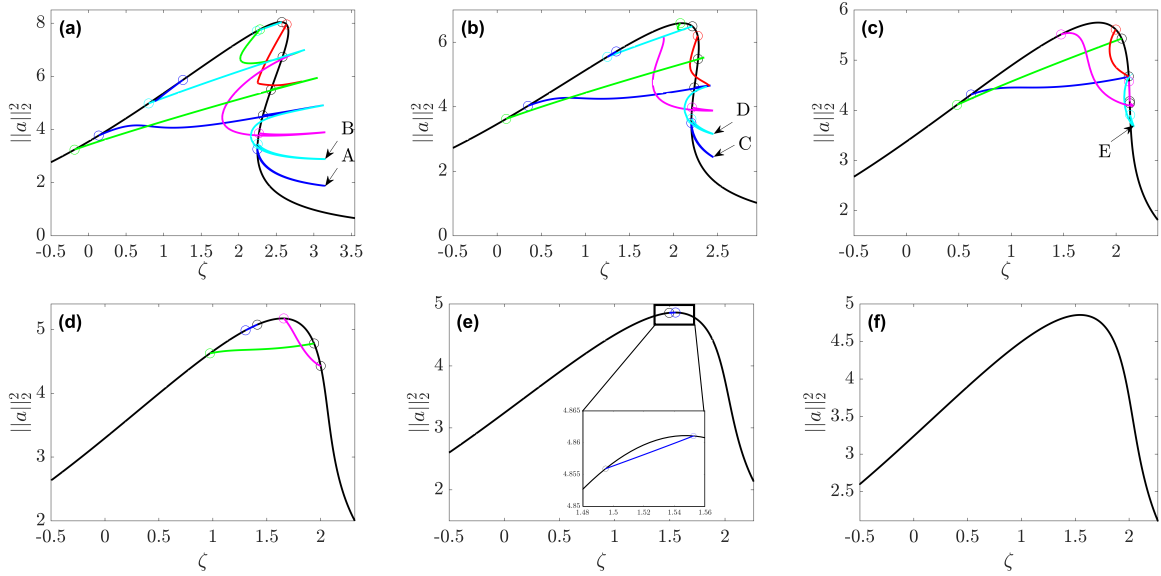

In Figure~\ref{solitons}(a), the solutions corresponding to the turning points A,C
in Figure~\ref{bif_dig} of the curve of 1-solitons are shown. Additionally, the 1-soliton at
the turning point of the corresponding branch for $\kappa=0.025$ is depicted. In
Figure~\ref{solitons}(b) the turning points B, D, E of the curve of 2-solitions are shown for different
values of the nonlinear damping coefficient. It becomes apparent that the solitons flatten as $\kappa$ increases.

\begin{figure}[H]
\centering
\begin{minipage}[t]{0.48\columnwidth}
\includegraphics[width=\textwidth]{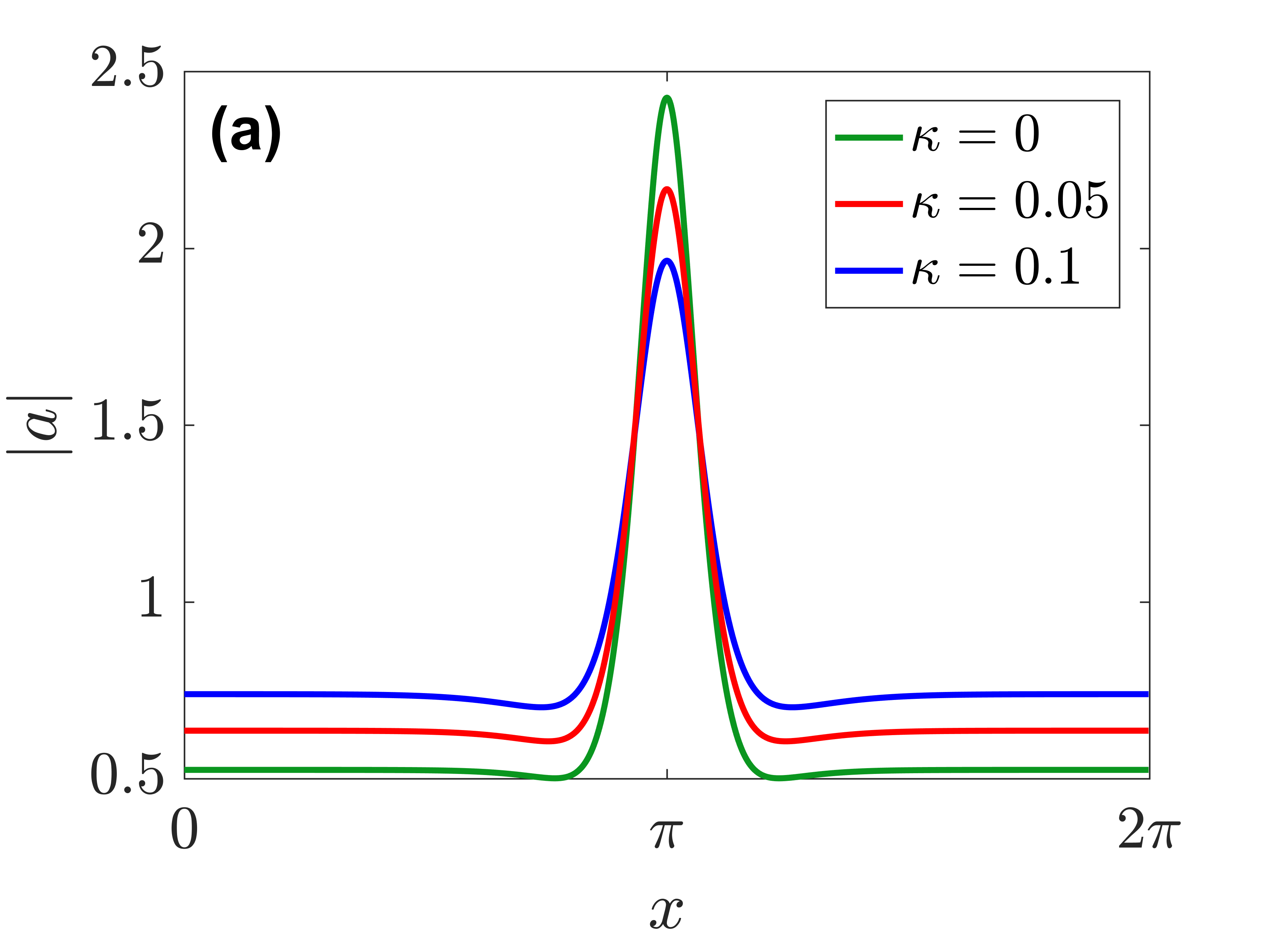}
 \end{minipage} 
 \hfill
 \begin{minipage}[t]{0.48\columnwidth}
  \includegraphics[width=\textwidth]{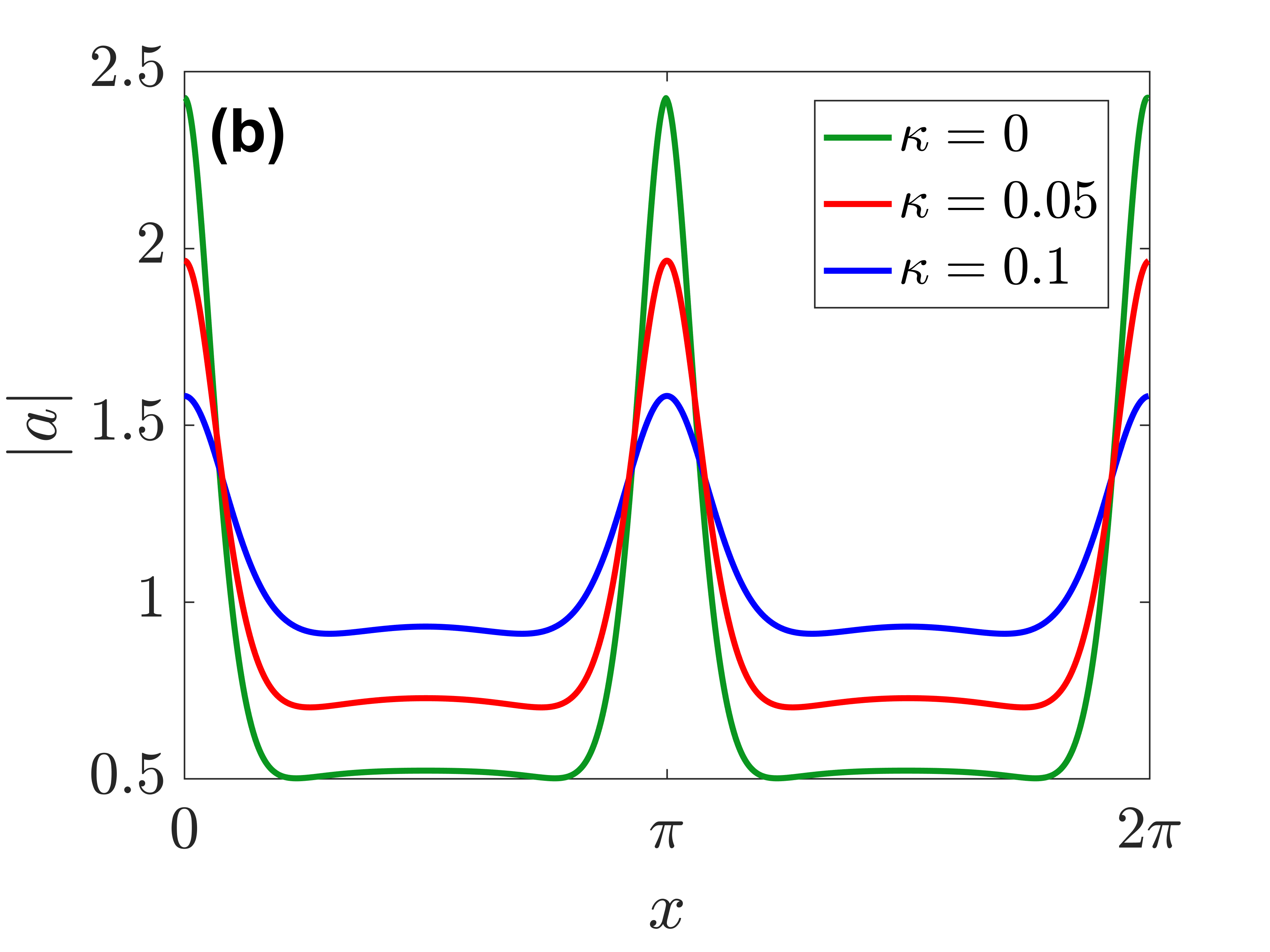} 
 \end{minipage}
\caption{Subfigure (a) shows 1-solitons and subfigure (b) 2-solitons of \eqref{eq:LL_complex} for increasing values of $\kappa$.}
 \label{solitons}
\end{figure}
  
  Since Theorem~\ref{thm:bifurcations} only addresses the  
  occurence and disappearance of bifurcations, it does not answer the question what happens to the entire
  set of solutions when $\kappa$ increases. This is answered in our next two results: all nontrivial
  solutions disappear for $\kappa$ beyond a certain positive threshold. A first threshold for nonexistence of
 nontrivial solutions is given by the following result.

\begin{thm} \label{thm:constancy}
Let $d\neq 0$, $\kappa>0$, $\zeta,f\in\RR$ and let $\kappa^\star$ be given by
\begin{align*}
  \kappa^\star:= 6\sqrt{6}\left(1+2\pi^2f^2\abs{d}^{-1}\right)^3 f^2.
\end{align*}
Then all solutions of \eqref{eq:LL_complex} are constant provided $\kappa>\kappa^\star$.
\end{thm}

 A second threshold may be obtained by studying the time-dependent Lugiato-Lefever equation \eqref{eq:LLE_timedependent}. 
 Modifying slightly the proof by Jahnke, Mikl and Schnaubelt~\cite{jami:14} for \eqref{eq:LLE_timedependent_standard} we first derive the global
 well-posedness of the initial value problem for  \eqref{eq:LLE_timedependent}  with  initial data $a(0)= \phi\in H^4(\T)$.
 In~\cite{jami:14} the corresponding well-posedness result for $\kappa=0$ is based on the observation that
 the flow remains bounded in $L^2(\T)$ and that the $H^1(\T)$-norm grows at most like $\sqrt{t}$ as
 $t\to\infty$. It is not known whether infinite time blow-up or convergence occurs in this case.   
 We show that for sufficiently strong nonlinear damping $\kappa\geq \frac{1}{\sqrt 3}$ the
 solutions converge to a constant solution regardless of the initial datum. 
  
  \begin{thm}\label{thm:flow}
    Let $d\neq 0$ $\zeta, f\in\RR$ and $\kappa\geq \frac{1}{\sqrt 3}$. If
    $a(0)=\phi\in H^4(\T)$ then the solution of~\eqref{eq:LLE_timedependent} is in $C(\R_+;H^4(\T))$
    and converges in $H^1(\T)$ to a constant as $t\to\infty$. In particular, all solutions
    of~\eqref{eq:LL_complex} are constant.
  \end{thm}
  
  Combining Theorem~\ref{thm:constancy} and Theorem~\ref{thm:flow} we obtain that for 
  $\kappa>\min\{\kappa^*,\frac{1}{\sqrt 3}\}$ only constant solutions exist. Notice that all weak solutions
  of~\eqref{eq:LL_complex} are smooth and in particular lie in $H^4(\T)$.  
  Actually we can also prove convergence results for smaller $\kappa$ assuming that $\|\phi_x\|_2$ is not too
  big. We refer to Lemma~\ref{lem:flow} for details.
  
  \medskip 
  
  Finally we discuss the effect of nonlinear damping to the Lugiato-Lefever equation on the real line in the case of anomalous dispersion $d>0$. In this
  case the problem reads
  \begin{align}\label{eq:LL_realline}
     -da'' -(\i-\zeta) a -(1+\i\kappa|a|^2)a + \i f=0  \quad\text{on }\R,\qquad a'(0)=0  
  \end{align}
  and we are interested in even homoclinic solutions. More precisely, the solutions we will
  find have the form  $a=\tilde{a}+a^\infty$ where $a^\infty\in\C$ and $\tilde{a}\in H^2(\RR)$.  
  This is a valid approach, since highly localized solutions of \eqref{eq:LL_realline} serve as good
  approximations for solutions of \eqref{eq:LL_complex}, cf. \cite{Herr2013}.
  Using a suitable singular rescaling of the problem as well as the Implicit Function Theorem, we prove
  the existence of large solutions of~\eqref{eq:LL_realline} for large parameters $\zeta$ and $f$ and small
  nonlinear damping $\kappa$.
  
  \begin{thm}  \label{thm:continuation}
    Let $d,\tilde\zeta>0$ and  $0< |\tilde f|<\frac{2\sqrt 3}{9}\tilde \zeta^{3/2}$. 
    Then for all $\varepsilon,\kappa>0$ sufficiently small there are two even homoclinic solutions $a_{\varepsilon,\kappa}$ of
    \eqref{eq:LL_realline} with $\zeta=\tilde{\zeta}\varepsilon^{-1},f=\tilde{f}\varepsilon^{-3/2}$ satisfying 
    $\norm{a_{\varepsilon,\kappa}-\lim_{\abs{x}\to\infty}a_{\varepsilon,\kappa}(x)}_{H^2}\to\infty$ as $\varepsilon\to 0$ uniformly with respect to $\kappa$. 
  \end{thm}
  \begin{remark} 
    The above theorem guarantees the existence of $\kappa_0, \epsilon_0>0$ depending on $d,\tilde \zeta, \tilde f$ with the property that 
    for $0<\kappa<\kappa_0$ and $0<\varepsilon<\varepsilon_0$ the parameter triple $(\zeta,f,\kappa)$ with $\zeta=\tilde{\zeta}\varepsilon^{-1}$ and $f=\tilde{f}\varepsilon^{-3/2}$ allows for a  localized solution of \eqref{eq:LL_realline}. For fixed $\kappa\in (0,\kappa_0)$ let us take $\varepsilon_0=\varepsilon_0(d,\tilde \zeta,\tilde f,\kappa)$ to be the largest value with the above property. Then we can consider the curve $(0,\varepsilon_0)\ni \varepsilon \mapsto (\tilde{\zeta}\varepsilon^{-1}, \tilde{f}\varepsilon^{-3/2})$ in the $(\zeta,f)$-plane. By varying the parameters $\tilde\zeta$ and $\tilde f$ these curves cover regions in the $(\zeta,f)$-plane, such that above the lower envelope $(\tilde{\zeta}\varepsilon_0^{-1},\tilde{f}\varepsilon_0^{-3/2})$ localized solutions of \eqref{eq:LL_realline} exist.
  \end{remark}
  
  The practical applicability of Theorem~\ref{thm:continuation} is demonstrated in the following. We have
  used the idea of the proof of the theorem as the basis for a numerical continuation method with
  \texttt{pde2path}.   This is done by replacing the real line with the interval $[0,\pi]$ and by considering
  the rescaled version \eqref{eq:onR} of the Lugiato-Lefever equation on $[0,\pi]$ with Neumann boundary
  conditions at the endpoints. Then, for a given fixed value of $\tilde\zeta$ and  $\tilde f=\varepsilon=\kappa=0$ the approximate solution $\i\sqrt{2\tilde\zeta}\sech(x\sqrt{\tilde\zeta/d})$ is
  continued first in $\tilde f$, then in $\varepsilon$ and finally in $\kappa$. Rescaling $a(x) =\varepsilon^{-1/2}u(\varepsilon^{-1/2}x)$ we obtain a function defined on
  $[0,\sqrt\varepsilon \pi]$ that we extend as a constant to
  $[\sqrt{\varepsilon}\pi,\pi]$. The resulting function is mirrored on the vertical axis and shifted
  by $\pi$ so that an approximate $2\pi$-periodic solution of \eqref{eq:LL_complex} for parameter values 
  $(\zeta,f)=(\tilde\zeta\varepsilon^{-1},\tilde
  f\varepsilon^{-3/2})$ is found. Refining this solution with a Newton
  step yields a periodic soliton solution $a$ solving \eqref{eq:LL_complex} on $[0,2\pi]$ for the
  parameters $(\zeta,f,\kappa)$. As an example, for fixed $d=0.1,\tilde \zeta=5$ we initially set $\tilde f=\varepsilon=\kappa=0$, and first continued the $\sech$-type soliton with respect to $\tilde f\in [0,2.9]$. For fixed $\tilde f=2.9$ the continuation is then done with respect to $\varepsilon\in [0, 0.5]$. Fixing both $\tilde f=2.9$ and $\varepsilon=0.5$ the final continuation is done in $\kappa$, and for three different values of
  $\kappa$ the resulting solutions are shown in Figure~\ref{solutions_cont}. With $\varepsilon=0.5$ the corresponding detuning and forcing values are $\zeta=\tilde\zeta\varepsilon^{-1}=10$ and $f=\tilde f\varepsilon^{-3/2}=8.20$.
  
  \begin{figure}[H]
\includegraphics[width=0.5\textwidth]{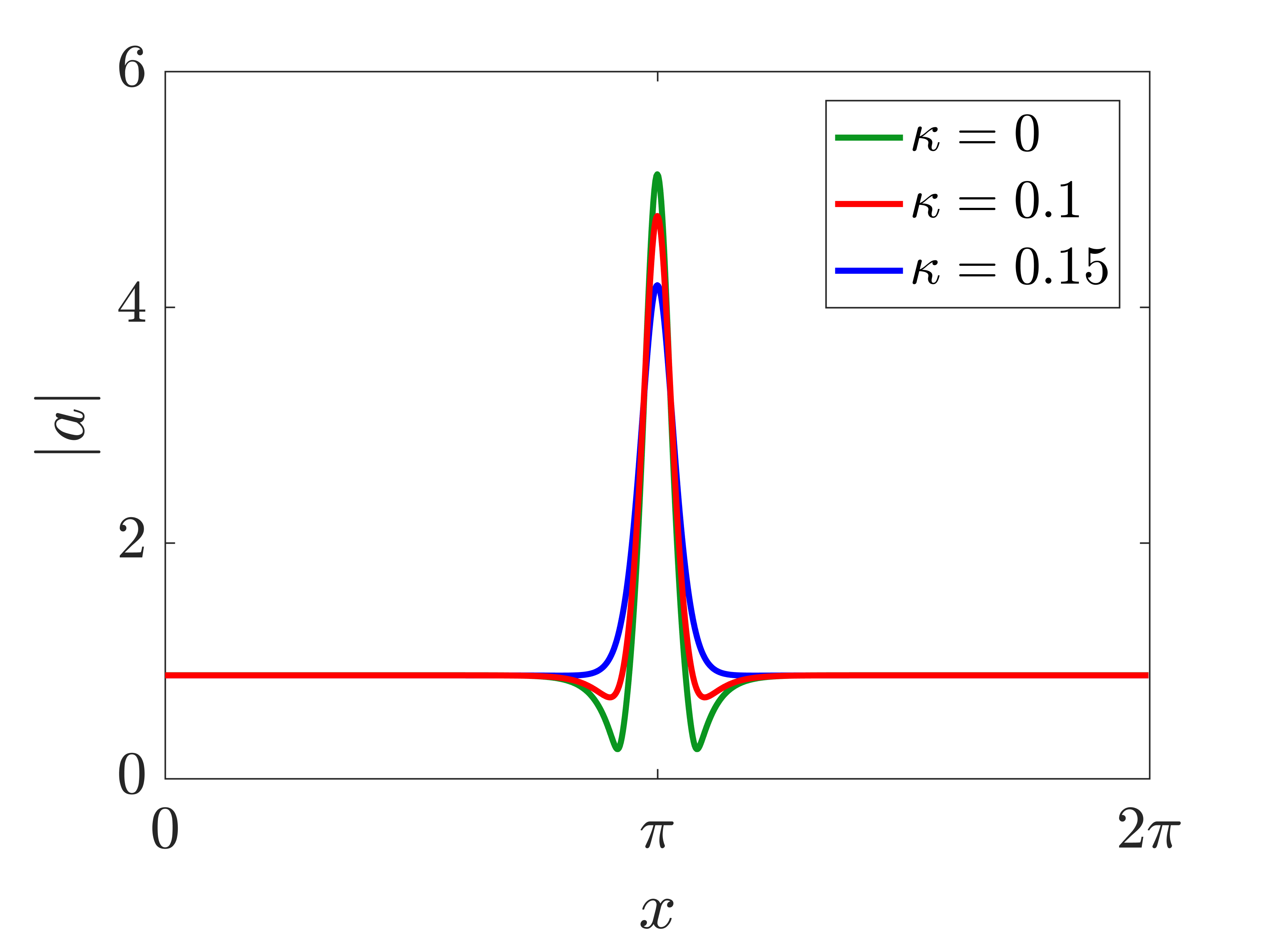}
\caption{Solutions of \eqref{eq:LL_realline} for  $d=0.1$, $\zeta=10$, $f=8.20$, and three different values of $\kappa$.}
\label{solutions_cont}
\end{figure}
  
  One might ask if a similar result for heteroclinic solutions in the case of normal dispersion $d<0$ could
  be achieved. In Section~\ref{sec:continuation}  we will point out that this cannot be done by our
  continuation method. The above result is of perturbative nature and therefore does not reveal whether
  nontrivial solutions of~\eqref{eq:LL_realline} have to disappear for large nonlinear damping $\kappa>0$ as
  it was shown in Theorem~\ref{thm:constancy} and Theorem~\ref{thm:flow} for the case of $2\pi$-periodic
  solutions of \eqref{eq:LL_complex}. The proofs of both theorems make use of the boundedness of $[0,2\pi]$
  in an essential way. Since we do not know how to adapt these results to
  solutions on $\R$ we have to leave this question open. 
  
\section{Proof of Theorem~\ref{thm:bifurcations}} \label{sec:bifurcation}
  
%

This section is structured according to the results in Theorem~\ref{thm:bifurcations}.

\subsection{Proof of (i).} Here we determine the curve of trivial solutions.

\begin{lem} \label{parametrization}
Let $\tau\in (0,1)$ be the unique value such that $\tau(1+\kappa f^2\tau)^2=1$. For $t\in (-\sqrt{\tau}, \sqrt{\tau})$ define 
$$
      A(t) := t\Big(\frac{1+4\kappa f^2\tau +3\kappa^2f^4\tau^2
        + t^2  (  -3\kappa^2f^4\tau - 2\kappa f^2) + t^4 \kappa^2 f^4}{\tau-t^2}\Big)^{1/2}.
$$
Then $t\mapsto (\zeta(t),a_0(t))$ parametrizes the curve of trivial solutions with 
\begin{align*}
\zeta(t) &:= f^2(\tau-t^2)  +  A(t), \\
a_0(t) &:= f(\tau-t^2) \big(1+\kappa f^2(\tau-t^2)  - \i A(t) \big).
\end{align*}
\end{lem} 

\begin{remark} 
  The curve $(\zeta,a_0):(-\sqrt{\tau},\sqrt{\tau})\to \R\times\R^2$ is smooth and unbounded in the
  $\zeta$-component. The same is true if we consider $(\zeta,a_0)$ as a map from $(-\sqrt{\tau},\sqrt{\tau})$
  into $\R\times H^2(\T)$. This is the claim of part (i) of Theorem~\ref{thm:bifurcations}.
\end{remark}

%
\begin{proof}
   Constant solutions $(a_0,\zeta)$ of~\eqref{eq:LL_complex} satisfy
    \begin{equation} \label{eq:const_solutions_I}
      (\zeta-\i)a_0-(1+i\kappa)|a_0|^2a_0+\i f = 0
    \end{equation}
    and in particular 
    \begin{equation}\label{eq:const_solutions_II}
      |a_0|^2\big( (\zeta-|a_0|^2)^2+(1+\kappa|a_0|^2)^2\big) = f^2.
    \end{equation}
    Let us successively parametrize $|a_0|^2$, $\zeta$ and $a_0$. Since $(\zeta-|a_0|^2)^2\geq 0$ we obtain from \eqref{eq:const_solutions_II} that
    \begin{equation}\label{eq:const_solutions_III}
      0<|a_0|^2f^{-2}\leq \tau, 
    \end{equation} 
    for $\tau\in (0,1)$ as in the statement of the lemma. Equation~\eqref{eq:const_solutions_II} suggests the
    following parametrization of $|a_0|^2$ by $t\mapsto |a_0|^2(t) :=f^2(\tau-t^2)$ for $t\in
    (-\sqrt{\tau},\sqrt{\tau})$. The sign of $t$ is chosen according to $\sign(t)=\sign(\zeta-|a_0|^2)$. Due
    to \eqref{eq:const_solutions_II}, the value $\zeta$ can be written as follows:
    \begin{align*}
      \zeta
      &= f^2(\tau-t^2) + \zeta-|a_0|^2 \\ 
      &= f^2(\tau-t^2)  + \sign(t) |\zeta-|a_0|^2|  \\
      &\stackrel{\eqref{eq:const_solutions_II}}{=} f^2(\tau-t^2)  + \sign(t) \sqrt{f^2|a_0|^{-2}-
      (1+\kappa|a_0|^2)^2}.
    \end{align*}
    Inserting the parametrization of $|a_0|^2(t)$ yields the following parametrization of $\zeta$
    \begin{align*} 
    \zeta(t) 
      = f^2(\tau-t^2)  + \sign(t)\sqrt{ \frac{1}{\tau-t^2} - (1+\kappa f^2(\tau-t^2))^2 } 
      = f^2(\tau-t^2)  +  A(t)
    \end{align*}
    Next we rearrange \eqref{eq:const_solutions_I} to express $a_0$ in terms of $f, \kappa, \zeta, |a_0|^2$ and use \eqref{eq:const_solutions_II} to find 
    \begin{align*}
      a_0
      &= \frac{\i f}{|a_0|^2-\zeta+i(1+\kappa|a_0|^2)}  \\
      &= \frac{-\i f(\i(1+\kappa|a_0|^2)+\zeta-|a_0|^2)}{(1+\kappa
      |a_0|^2)^2+(\zeta-|a_0|^2)^2}  \\
       &\stackrel{\eqref{eq:const_solutions_II}}{=}\frac{|a_0|^2}{f}\left(1+\kappa|a_0|^2+
       \i(|a_0|^2-\zeta)\right).
    \end{align*}
    If we insert $|a_0|^2(t)=f^2(\tau-t^2)$ and $\zeta(t)=f^2(\tau-t^2)+A(t)$ into the previous expression we
    finally arrive at 
    $$
      a_0(t) = f(\tau-t^2) \left(1+\kappa f^2(\tau-t^2)  - \i A(t) \right). 
    $$
\end{proof}
    
\subsection{Proof of (ii) -- necessary and sufficient conditions for bifurcation}
 
In order to prove (ii) we need the following preliminary result, which is a generalization of Proposition~4.3
in \cite{MaRe_aprioribounds}. It provides the necessary condition for bifurcation.

   \begin{prop} \label{prop:bifcondition}
    All bifurcation points $(\zeta,a_0)$ for \eqref{eq:LL_complex} with respect to the curve of trivial solutions satisfy 
    \begin{equation}\label{eq:nec_cond_bifurcation}
      (\zeta+dk^2)^2-4|a_0|^2(\zeta+dk^2)+3(1+\kappa^2)|a_0|^4+4\kappa|a_0|^2+1 = 0
    \end{equation}
    for some $k\in\N$. In other words, one of the two numbers $k_{1,2}$ from \eqref{integer} needs to be in $\N$. 
  \end{prop}
  
  \begin{remark} \label{nicht_k=0} 
     We exclude the case $k_1=0$ or $k_2=0$ in the bifurcation condition \eqref{eq:nec_cond_bifurcation}. 
     It happens exactly at the turning points of the curve of trivial solutions and corresponds to the
     non-injectivity of $\zeta(t)$.  Since it creates only artificial bifurcation points as explained in
     Section~4.2 in \cite{MaRe_aprioribounds}, we omit it.
  \end{remark}
  
  \begin{proof} By the implicit function theorem we know that a necessary condition for bifurcation is that the linearized operator
  \begin{equation} \label{def_L}
L = -d \frac{d^2}{dx^2} - (\i-\zeta)-Dg(a_0): H^2(\T)\to L^2(\T)
\end{equation}
has a nontrivial kernel. Here $g(a)=(1+\i\kappa)|a|^2a-\i f$ stands for the nonlinearity and $Dg(a)z := \frac{d}{dt} g(a+tz)|_{t=0}= 2(1+\i\kappa)|a|^2z+(1+\i\kappa)a^2\bar z$ with $a,z\in \C$ for the derivative of $g$ at $a$. The derivative $Dg(a)$ can also be written in the form
\begin{equation} \label{Dg}
Dg(a) z = \begin{pmatrix} \Re(a^2)+2|a|^2-\kappa \Im(a^2) & \Im(a^2)-2\kappa |a|^2+\kappa\Re (a^2) \\
\Im(a^2) +\kappa\Re(a^2)+2\kappa |a|^2 &  2|a|^2-(\Re a^2)+\kappa\Im(a^2)
\end{pmatrix} 
\begin{pmatrix} \Re z \\ \Im z \end{pmatrix}.
\end{equation} 
Since $L$ is a Fredholm operator, the space $\Kern L$ is finite dimensional, and the adjoint operator
 \begin{equation} \label{def_Lstar}
L^\ast= -d \frac{d^2}{dx^2} + (\i+\zeta)-\overline{Dg(a_0)}: H^2(\T)\to L^2(\T)
\end{equation} 
has a kernel with the same finite dimension as $\Kern L$. Any element $\phi\in\Kern L$ can be expanded in the form 
$\phi(x) = \sum_{l\in \Z} \alpha_l \e^{\i lx}$. The condition that $\phi\in\Kern L$ means that there is at
least one integer $k\in \Z$ such that $L (\alpha \e^{\i kx})=(dk^2 -\i+\zeta-Dg(a_0))\alpha \e^{\i kx} =0$
for some $\alpha\in \C\setminus\{0\}$. In other words, $dk^2$ is an eigenvalue of the matrix 
$$ 
  N= Dg(a_0) + \begin{pmatrix} -\zeta & -1 \\ 1 & -\zeta \end{pmatrix}
$$
with $Dg(a_0)$ in matrix representation given by \eqref{Dg}. Non-zero elements in $\Kern L$ exist if
$\det(-dk^2\Id+N)=0$ and computing this determinant yields \eqref{eq:nec_cond_bifurcation}. Solving for $k$ leads to $k_{1,2}$ given by \eqref{integer}. Likewise, non-zero elements in $\Kern L^\ast$ exist if $\det(-d{\tilde k}^2\Id+N^T)=0$ for some integer $\tilde k\in \N_0$.
 Solving $\det(-d{\tilde k}^2\Id+N^T)=\det(-d{\tilde k}^2\Id+N)=0$ leads to the same formula~\eqref{integer}
 as for $k$. Consequently, \eqref{integer} is equivalent to both $L$ and $L^\ast$
 having nontrivial kernels. If neither $k_1$ or $k_2$ are in $\N$ then $\Kern L=\Kern L^\ast=\{0\}$, and in
 this case the implicit function theorem, cf. \cite{Kielh_bifurcation_theory}[Theorem I.1.1], implies that
 solutions nearby the point $(\zeta,a_0)$ are unique, i.e., trivial, and hence $(\zeta_,a_0)$ cannot be a
 bifurcation point. Therefore, $k_1$ or $k_2$ in $\N$ is a necessary condition for bifurcation. 
\end{proof}

\subsection{Proof of (ii) -- simplicity of the kernel of the linearization}
Notice that $\Kern L$ is either two-dimensional or four-dimensional, since $\alpha \e^{\i kx}$ belonging to
$\Kern L$ always implies that $\alpha \e^{-\i kx}$ also belongs to $\Kern L$.  The two-dimensional case
happens if exactly one of the two numbers $k_{1,2}$ from~\eqref{integer} is an integer and the
four-dimensional case happens if both $k_1, k_2\in \N$.

\medskip

In order to achieve simple instead of multiple eigenvalues we need to change the setting
for~\eqref{eq:LL_complex} by additionally requiring $a'(0)=0$, i.e., solutions need to be even around $x=0$.
Together with $2\pi$-periodicity this implies $a'(\pi)=0$, i.e., we   consider~\eqref{eq:LL_complex}
with vanishing Neumann boundary conditions at $x=0$ and $x=\pi$. If we define $H^2_{even}(\T)$ and
$L^2_{even}(\T)$ as the subspaces of $H^2(\T)$ and $L^2(\T)$ with even symmetry around $x=0$ then $L, L^\ast:
H^2_{even}(\T) \to L^2_{even}(\T)$ are again Fredholm operators, and Propositon~\ref{prop:bifcondition} still holds. 
In this way we halve the dimension of $\Kern L$ for every $k$ satisfying~\eqref{eq:nec_cond_bifurcation} since
instead of both $\alpha \e^{\i kx}$ and $\alpha \e^{-\i kx}$ only $\alpha \cos(kx)$ remains in the kernel of
$L$. In particular, we get a one-dimensional kernel of $L$ if and only if exactly one the numbers $k_{1,2}$
from~\eqref{integer} belongs to $\N$. The same is true for the kernel of $L^\ast$.

\subsection{Proof of (ii) - computing the kernel of the linearization} Under the condition
that exactly one of the numbers $k_{1,2}$ from~\eqref{integer} belongs to $\N$ let us compute $\Kern L$ and
$\Kern L^\ast$. To describe the matrix $N-dk^2\Id$ let us introduce the real numbers $\alpha_j,
\tilde\alpha_j, \alpha_j^\ast, \tilde \alpha_j^\ast$ for $j=1,2$ as follows
\begin{multline} \label{volle_matrix}
N-dk^2\Id = \begin{pmatrix}
-\alpha_2 & \alpha_1  \\
\tilde\alpha_2 & -\tilde \alpha_1
\end{pmatrix} = \begin{pmatrix}
-\tilde\alpha_2^\ast & \alpha_2^\ast \\
\tilde \alpha_1^\ast & -\alpha_1^\ast
\end{pmatrix} \\
= \begin{pmatrix}
-\zeta-dk^2+(\Re a_0)^2+2|a_0|^2 -\kappa \Im(a_0^2) & \Im(a_0^2)-1-2\kappa|a_0|^2+\kappa\Re( a_0^2) \\
\Im(a_0^2)+1+\kappa \Re(a_0^2)+2\kappa|a_0|^2 & -\zeta-dk^2+2|a_0|^2-\Re(a_0^2)+\kappa \Im(a_0^2)
\end{pmatrix}.
\end{multline}
In the matrix $N-dk^2\Id$ the off-diagonal elements have the property that
$$
\alpha_1 <\Im(a_0^2)<\tilde\alpha_2 
$$
and hence they cannot be zero simultaneously. Therefore, if $\Im(a_0^2)\leq 0$ we can define  
\begin{equation} \label{def_alpha}
\alpha:= (\alpha_1,\alpha_2)^T, \quad 
\alpha^\ast:= (\alpha_1^\ast,\alpha_2^\ast)^T
\end{equation}
and obtain eigenvectors of $N-dk^2\Id$, $N^T-k^2\Id$, respectively, so that $\Kern L = \spann\{\alpha\e^{\i
kx}\}$, $\Kern L^\ast = \spann\{\alpha^\ast\e^{\i kx}\}$. Likewise, if $\Im(a_0^2)\geq 0$ then  
\begin{equation} \label{def_tildealpha}
\tilde\alpha:= (\tilde\alpha_1,\tilde\alpha_2)^T, \quad 
\tilde\alpha^\ast:= (\tilde\alpha_1^\ast,\tilde\alpha_2^\ast)^T
\end{equation}
are the eigenvectors of $N-dk^2\Id$, $N^T-k^2\Id$ leading to $\Kern L = \spann\{\tilde\alpha\e^{\i kx}\}$,
$\Kern L^\ast = \spann\{\tilde\alpha^\ast\e^{\i kx}\}$.

\subsection{Proof of (ii) -- tangent direction to the trivial branch of solutions} Let us assume that the
curve of trivial solutions of \eqref{eq:LL_complex} is parameterized by $t \mapsto (\zeta(t), a_0(t))$ as in Lemma~\ref{parametrization}, 
 and that $(\zeta,a_0)=(\zeta(t_0),a_0(t_0))$ is a specific bifurcation point. Let us compute the tangent
$(\dot \zeta, \dot a_0) = \frac{d}{dt} (\zeta(t), a_0(t)|_{t=t_0}$. 
As explained in Remark~\ref{nicht_k=0} we can ignore turning points where $\dot\zeta=0$.    
 Differentiating the equation $(\mathrm{i}-\zeta(t))a_0(t)+g(a_0(t))=0$ with respect to $t$ and evaluating the
derivative at $t_0$ we get 
$$ 
  \bigl(Dg(a_0)+\i-\zeta\bigr)\dot a_0 = \dot \zeta a_0.
$$
Inserting $Dg(a_0)z = (1+\i\kappa)(2|a_0|^2z+a_0^2\bar z)$ we find
$$
(2(1+\i\kappa)|a_0|^2+\i-\zeta)\dot a_0+ (1+\i\kappa)a_0^2\overline{\dot a_0} = \dot \zeta a_0
$$
and hence
\begin{equation} \label{tangent}
\dot a_0 = \tau \dot\zeta a_0 \quad\mbox{with}\quad \tau =
\frac{(1-3\i\kappa)|a_0|^2-\zeta-\i}{3(1+\kappa^2)|a_0|^4+4(\kappa-\zeta)|a_0|^2+\zeta^2+1}.
\end{equation}

\subsection{Proof of (ii) -- sufficient condition for bifurcation} 
According to the Crandall-Rabinowitz theorem, see \cite{CrRab_bifurcation} or
\cite{Kielh_bifurcation_theory}[Theorem I.5.1], two conditions are sufficient for bifurcation. The first is that $\Kern L$ is simple, i.e. one-dimensional. Above we proved this to hold provided
$k_1\in \N$, $k_2\not \in \N$ or vice versa with $k_{1,2}$ from \eqref{integer}. 
In the following we write $k$ for the one which is the integer. In view of the statement of (ii) it
therefore remains to show that the second condition, the so-called transversality condition, is satisfied
provided~\eqref{transversality} holds. To verify this we bring our problem into the form
used in \cite{CrRab_bifurcation}. Nontrivial solutions of~\eqref{eq:LL_complex}, which are even around $x=0$
may be written as $a(\cdot)=a_0(t)+b(\cdot)$ with $b'(0)=b'(\pi)=0$. From \eqref{eq:LL_complex} we derive the
equation for the function $b$ in the form
\begin{equation} \label{ll_b}
F(t,b) := -db'' -(\i-\zeta(t))(a_0(t)+b)-g(a_0(t)+b)=0
\end{equation}
where $F: \R\times H^2_{even}(\T)\to L^2_{even}(\T)$. 
Notice that $F(t,0)=0$ for all $t$, i.e., the curve of trivial solutions $(\zeta(t),a_0(t))$ for
\eqref{eq:LL_complex} has now become the line of zero solutions $(t,0)$ for \eqref{ll_b}. Let us write
$D^2_{b,t} F(t_0,0)$ for the mixed second derivative of $F$ with respect to $(t,\zeta)$ at the point
$(t_0,0)$. According to \cite{CrRab_bifurcation}, the transversality condition is expressed by 
$$ 
  D^2_{b,t} F(t_0,0)\phi \not \in \rg D_b F(t_0,0), 
$$
with $\phi$ such that $\Kern D_b F(t_0,0)=\spann\{\phi\}$. In our case $D_b F(t_0,0)=L$, where $L$ is the linearized operator given in \eqref{def_L}. By the Fredholm alternative, $\rg L= (\Kern L^\ast)^\perp=\spann\{\phi^\ast\}^\perp$, and $\phi(x)=\alpha\cos(kx)$, $\phi^\ast(x)= \alpha^\ast\cos(kx)$ if $\Im(a_0^2)\leq 0$, cf. \eqref{def_alpha}, or 
$\phi(x)=\tilde\alpha\cos(kx)$, $\phi^\ast(x)= \tilde\alpha^\ast\cos(kx)$ with $\tilde\alpha,
\tilde\alpha^\ast$ if $\Im(a_0^2)\geq 0$, cf. \eqref{def_tildealpha}. The components of $\alpha, \alpha^\ast$
and $\tilde\alpha, \tilde\alpha^\ast$ can be read from \eqref{volle_matrix}. Since orthogonality of two
functions $u,v$ in the real Hilbert space $L^2_{even}(\T)$ means vanishing of the inner product $\langle
u,v\rangle = \Re\int_0^\pi u(x)\bar v(x)\,dx$, we find that transversality is expressed as
\begin{equation} \label{trans_expressed}
\langle D^2_{b,t} F(t_0,0)\phi,\phi^\ast\rangle= \Re \int_0^\pi \bigl(D^2_{b,t} F(t_0,0)\phi\bigr) \overline{\phi^\ast}\,dx \not =0.
\end{equation}
Using $D^2g(a_0)(z,w) = 2(1+\i\kappa)(\bar a_0 zw+a_0z\bar w+a_0\bar z w)$ we find for the second derivative
\begin{equation} \label{second_derivative}
\begin{aligned}
D^2_{b,t} F(t_0,0)\phi &= \dot\zeta\phi - D^2g(a_0) (\phi,\dot a_0) \\
&= \dot\zeta\phi - 2(1+\i\kappa) \left(\bar a_0\phi\dot a_0 +a_0\phi\overline{\dot a_0} +a_0\bar \phi\dot a_0\right)
\end{aligned}
\end{equation}
with $\dot a_0 = \tau \dot\zeta a_0$, $\tau$ from \eqref{tangent}. As explained in Remark~\eqref{nicht_k=0}
we can ignore the turning points where $\dot\zeta=0$. Hence, inserting \eqref{second_derivative} into the
transversality condition \eqref{trans_expressed} we get in case $\Im(a_0^2)\leq 0$
\begin{equation} \label{trans_expressed2}
\Re \Bigl(\alpha\overline{\alpha^\ast} -2(1+\i\kappa)(2\Re \tau |a_0|^2 \alpha\overline{\alpha^\ast}+\tau a_0^2\bar \alpha \overline{\alpha^\ast} \bigr)\Bigr)\not =0
\end{equation}
and in case  $\Im(a_0^2)\geq 0$ we replace $\alpha, \alpha^\ast$ by $\tilde\alpha, \tilde\alpha^\ast$. Let us first consider the case $\Im(a_0^2)\leq 0$. Here we obtain 
\begin{equation} \label{produkt1}
\begin{aligned}
\alpha\overline{\alpha^\ast} =&(\alpha_1+\i\alpha_2)( \alpha_1^\ast-\i\underbrace{\alpha_2^\ast}_{=\alpha_1}) = \alpha_1(\alpha_1^\ast+\alpha_2)+\i(\underbrace{\alpha_2\alpha_1^\ast}_{=\alpha_1\tilde\alpha_1^\ast}-\alpha_1^2) \\
=& \alpha_1\Bigl(2\zeta_0+2dk^2 -4|a_0|^2+ \i(2+4\kappa|a_0|^2)\Bigr),
\end{aligned}
\end{equation}
Likewise, we use \eqref{volle_matrix} and $\det(N-dk^2\Id)=0$ to compute
\begin{equation} \label{produkt2}
\begin{aligned}
\bar\alpha\overline{\alpha^\ast}=& (\alpha_1-\i \alpha_2)(\alpha_1^\ast-\i \underbrace{\alpha_2^\ast}_{=\alpha_1})
= \alpha_1(\alpha_1^\ast-\alpha_2)-\i(\alpha_1^2+\underbrace{\alpha_2\alpha_1^\ast}_{=\alpha_1\tilde\alpha_1^\ast}) \\
=& \alpha_1 2(1-\i\kappa)\bar a_0^2.
\end{aligned}
\end{equation}
Taking the expressions for 
$\alpha\overline{\alpha^\ast}$ and $a_0^2\bar\alpha \overline{\alpha^\ast}$ into the transversality condition \eqref{trans_expressed2} finally leads to 
\begin{eqnarray}
\lefteqn{\Re \Bigl(\alpha\overline{\alpha^\ast}(1-4(1+\i\kappa)\Re \tau |a_0|^2)\Bigr) 
-\Re\Bigl(\bar \alpha \overline{\alpha^\ast} 2(1+\i\kappa) \tau a_0^2\Bigr)} \nonumber\\
&=& (1-4\Re \tau |a_0|^2)\Re(\alpha\overline{\alpha^\ast})+4\kappa\Re\tau |a_0|^2 \Im(\alpha\overline{\alpha^\ast})-4\alpha_1(1+\kappa^2)|a_0|^4\Re\tau  \label{tr} \\
&=&\alpha_1 \Bigl(2\zeta_0+2dk^2-4|a_0|^2-4\Re\tau |a_0|^2( 2\zeta_0+2dk^2-3|a_0|^2(1+\kappa^2)-2\kappa)\Bigr)\not =0. \nonumber
\end{eqnarray}
Since $\Im(a_0^2)\leq 0$ implies that $\alpha_1$ is non-zero, the non-vanishing of the expression in brackets amounts to (after inserting $\Re \tau$ from \eqref{tangent}) 
$$
(\zeta-dk^2)(|a_0|^4-3\kappa^2)+(\zeta^2+1)(\zeta+dk^2-2|a_0|^2) -4\kappa |a_0|^2(|a_0|^2+dk^2) \not =0.
$$
Using \eqref{integer} we obtain the transversality condition \eqref{transversality}.

\medskip

Changes in case $\Im(a_0^2)\geq 0$ amount to replacing $\alpha_1$ in \eqref{produkt1}, \eqref{produkt2} and \eqref{tr} by $\tilde\alpha_2$, which is non-zero in this case. Therefore, the final transversality condition \eqref{transversality} is the same as before.
   
\subsection{Proof of (iii) -- nonexistence of bifurcations}  
  We assume that bifurcation for~\eqref{eq:LL_complex} occurs at some trivial solution $(\zeta,a_0)$ so
  that the claim is proved once we show $\kappa\leq \kappa_*$.
    By Proposition~\ref{prop:bifcondition} we know that the quadratic equation in $\zeta+dk^2$ from
    \eqref{eq:nec_cond_bifurcation} holds for some $k\in\N_0$. In particular, the discriminant is nonnegative
    and we obtain
    \begin{equation}\label{eq:nec_cond_bifurcation_I}
      0 \leq (4|a_0|^2)^2 - 4\cdot (3(1+\kappa^2)|a_0|^4+4\kappa|a_0|^2+1) 
        = 4\big( (1-3\kappa^2)|a_0|^4-4\kappa|a_0|^2-1\big).
    \end{equation}  
    For $\kappa\geq \frac{1}{\sqrt 3}$ this inequality is unsolvable, so we necessarily have
    $\kappa\in [0,\frac{1}{\sqrt 3})$ as well as
    \begin{equation} \label{a_quadrat_unten}
      |a_0|^2\geq \frac{2\kappa+\sqrt{1+\kappa^2}}{1-3\kappa^2}.
    \end{equation}
    On the other hand, the inequality \eqref{eq:const_solutions_III} from the proof of (i) gives $|a_0|^2\leq
    f^2\tau$ where $\tau$ is the unique value such that $\tau(1+kf^2\tau)^2=1$. Therefore 
    \begin{equation} \label{def_tildetau}
      \tilde\tau := \frac{2\kappa+\sqrt{1+\kappa^2}}{(1-3\kappa^2)f^2}\leq \frac{|a_0|^2}{f^2}
      \leq \tau.
    \end{equation}
    Since $z\mapsto z(1+\kappa f^2z)^2$ is increasing on $[0,\infty)$, we deduce from the definition of
    $\tau$ the inequality 
    $$
      \tilde\tau (1+\kappa f^2 \tilde\tau)^2 
      \leq \tau(1+\kappa f^2 \tau)^2= 1.
    $$
		Inserting $\tilde\tau$ from \eqref{def_tildetau} this is equivalent to
    $$
    \frac{2\kappa+\sqrt{1+\kappa^2}}{(1-3\kappa^2)^3} (1-\kappa^2+\kappa\sqrt{1+\kappa^2})^2 \leq f^2,
    $$
    which implies $\kappa\leq \kappa_*$ by definition of $\kappa_*$. This finishes the proof of (iii). \qed


\section{Proof of Theorem~\ref{thm:constancy}}  

Following \cite{MaRe_aprioribounds} we first provide some a priori bounds in $L^\infty(\T)$ for
solutions of~\eqref{eq:LL_complex}. 

\begin{thm} \label{thm:a_priori_bounds}
Let $d\neq 0$, $\kappa>0$ and $\zeta,f\in\RR$. Then every solution $a\in C^2(\T)$
of \eqref{eq:LL_complex} satisfies
\begin{align} \label{l_unendlich_bound}
  \norm{a}_\infty\leq\left(1+2\pi^2f^2\abs{d}^{-1}\right)\min\left\{|f|,\left(\frac{|f|}{\kappa}\right)^{1/3}\right\}.
\end{align}
\end{thm}

\begin{remark}
One can obtain a more refined version of the bound \eqref{l_unendlich_bound} of the form 
$\norm{a}_\infty\leq\left(1+2\pi^2f^2\abs{d}^{-1}\right) C_\kappa$ where 
\begin{equation}\label{eq:defn_Ckappa}
  C_{\kappa} = \sqrt[3]{\frac{\abs{f}}{2\kappa}+\sqrt{\frac{f^2}{4\kappa^2}+\frac{1}{27\kappa^3}}}-
  \sqrt[3]{-\frac{\abs{f}}{2\kappa}+\sqrt{\frac{f^2}{4\kappa^2}+\frac{1}{27\kappa^3}}}.
\end{equation}
This follows from Cardano's formula applied to \eqref{basis_fuer_cardano}. In this paper we do not make
further use of the refined value of $C_\kappa$, since \eqref{l_unendlich_bound} already provides a meaningful
a priori bound both for small as well as for large values of $\kappa$. Indeed, as $\kappa\to 0^+$ the
$L^\infty$-bounds from (2) in~\cite{MaRe_aprioribounds} (valid for $\kappa=0$) are partially recovered.
\end{remark}

\begin{proof}
Let $a\in H^2(\T)$ be a solution of \eqref{eq:LL_complex}. Then we define the $2\pi$-periodic function $g:=-d\Im(a'\bar a)'$. Using \eqref{eq:LL_complex} we obtain
\begin{align}
g&=-d\Im(a''\bar a) \nonumber\\
&= \Im\Bigl((\i-\zeta)|a|^2+(1+\i\kappa)|a|^4-\i f\bar a\Bigr) \label{identity_g}\\
&=\abs{a}^2+\kappa\abs{a}^4-f\Re a. \nonumber
\end{align}
Using the fact that $g$ is $2\pi$-periodic together with H\"older's inequality we get from the previous identity
\begin{align}
  0=&\int_0^{2\pi}g\,dx=\int_0^{2\pi}(\abs{a}^2+\kappa\abs{a}^4-f\Re a)\, dx \nonumber \\
\geq& \kappa\norm{a}_4^4+\norm{a}_2^2-\sqrt{2\pi}\abs{f}\norm{a}_2 \label{basis_fuer_cardano}\\
\geq& \norm{a}_2\left(\frac{\kappa}{2\pi}\norm{a}_2^3+\norm{a}_2-\sqrt{2\pi}\abs{f}\right). \nonumber
\end{align}
Neglecting once the $\|a\|_2^3$ and once the $\|a\|_2$ term we obtain the $L^2$-bound
\begin{align}\label{eq:L2}
\norm{a}_2\leq \sqrt{2\pi} \tilde C_\kappa \mbox{ with } \tilde C_\kappa =
\min\left\{|f|,\left(\frac{|f|}{\kappa}\right)^{1/3}\right\}.
\end{align}

Next we derive a bound for $\|a'\|_2$. First, the differential equation \eqref{eq:LL_complex} yields the identity
\begin{align}
\norm{a'}_2^2=&\Re \int_0^{2\pi}\bigl(-\i da''-\i\zeta a+(\i-\kappa)|a|^2a+f\bigr)'\bar a'\,dx
 \nonumber \\
=&\Re\int_0^{2\pi} -\i d a'''\bar a'+\i(|a|^2a)'\bar a'-\kappa(|a|^2a)'\bar a'\,dx \nonumber\\
=&\Re \int_0^{2\pi} \i (|a|^2)' a\bar a'\,dx -\kappa \int_0^{2\pi} |a|^2 |a'|^2\,dx -\kappa\Re\int_0^{2\pi} (|a|^2)'a\bar a'\,dx \label{identity}\\
=&-\Im \int_0^{2\pi} (|a|^2)'a\bar a'\,dx -\kappa\int_0^{2\pi} |a|^2 |a'|^2\,dx -\frac{\kappa}{2}\int_0^{2\pi} (|a|^2)'(|a|^2)'\,dx \nonumber \\
\leq& -\Im \int_0^{2\pi} (|a|^2)'a\bar a'\,dx. \nonumber
\end{align} 
Next we set $G:=-d\Im(a'\bar a)=d\Im(\bar a' a)$ so that $G'=g$ as well as $G(0)=G(2\pi)$. 
Using the identity \eqref{identity_g} we get the pointwise estimate $g\geq-\frac{f^2}{4}$ on $[0,2\pi]$ from which we deduce
\begin{equation}
\label{eq:G}
\begin{aligned}
G(x)-G(0)&=\int_0^xg(t)\,dt\geq-\frac{\pi}{2}f^2\quad (x\in [0,2\pi])\quad\text{and} \\
G(x)-G(2\pi)&=-\int_x^{2\pi}g(t)\,dt\leq\frac{\pi}{2}f^2\quad (x\in [0,2\pi]).
\end{aligned}
\end{equation}
Using the definition of $G$ and \eqref{eq:G} we deduce from \eqref{identity}
\begin{align*}
\abs{d}\norm{a'}_2^2 
\leq & \left| d\Im\left(\int_0^{2\pi}(\abs{a}^2)'\bar
a'a\,dx\right) \right| 
= \left|\int_0^{2\pi}(\abs{a}^2)'G\,dx\right| \\
\leq &\int_0^{2\pi}(\abs{a}^2)'|G-G(0)|\,dx\\
\leq&\frac{\pi f^2}{2}\int_0^{2\pi}|(\abs{a}^2)'|\,dx=\pi f^2\int_0^{2\pi}\abs{a}\abs{a'}\,dx\\
\leq&\pi f^2\norm{a}_2\norm{a'}_2\\
\leq&\sqrt{2}\pi^{3/2}f^2 \tilde C_\kappa\norm{a'}_2
\end{align*}
with $\tilde C_\kappa$ from \eqref{eq:L2}. So we find
\begin{align}\label{eq:H1}
\abs{d}\norm{a'}_2\leq \sqrt{2}\pi^{3/2}f^2 \tilde C_\kappa.
\end{align}

\medskip

Finally, we combine the previous estimates for $\|a\|_2,\|a'\|_2$ to deduce an $L^\infty$-estimate. 
From~\eqref{eq:L2} we obtain that there is an $x_1\in [0,2\pi]$ satisfying
$\abs{a(x_1)}\leq \tilde C_\kappa$. Together with~\eqref{eq:H1} this implies
\begin{align} \label{eq:Linfty_estimate}
  \begin{aligned}
\norm{a}_\infty&\leq \abs{a(x_1)}+\norm{a-a(x_1)}_\infty\\
&\leq \tilde C_\kappa+\norm{a'}_1\\
&\leq \tilde C_\kappa+\sqrt{2\pi}\norm{a'}_2\\
&\leq\left(1+2\pi^2f^2\abs{d}^{-1}\right) \tilde C_\kappa.
  \end{aligned}
\end{align}
\end{proof}

With these bounds the constancy of solutions for large $\kappa$ is proved along the
lines of the proof of Theorem~2 in~\cite{MaRe_aprioribounds}. However, from a technical point of view, several partial
results from the proof presented in~\cite{MaRe_aprioribounds} break down and new difficulties have to be
overcome so that the proof given next contains several new aspects.  

\medskip

\noindent
\begin{proof}[Proof of Theorem~\ref{thm:constancy}] We equip the real Hilbert space $H^1(\T)$ with the inner product generated
by the norm 
\begin{equation}\label{eq:def_H1norm}
  \|\phi\|_{H^1}^2:= \gamma \norm{\phi'}_2^2+\norm{\phi}_2^2
   \qquad\text{for }\phi\in H^1(\T)
\end{equation} 
where $\gamma>0$ will be suitably chosen later. We observe that
a solution $a:[0,2\pi]\to\CC$ of \eqref{eq:LL_complex} is constant if and only if the function
$A=a'$ is trivial. Since $a$ solves \eqref{eq:LL_complex} the function $A$ is a $2\pi$-periodic solution of the differential equation
\begin{equation}\label{eq:bvp_A1A2}
 -dA''=(\i-\zeta)A+ 2(1+\i\kappa)|a|^2 A+(1+\i\kappa)a^2\bar A.
\end{equation}
We introduce the differential operator $L_\kappa: H^2(\T)\subset L^2(\T)\to L^2(\T)$ by
\begin{equation}\label{eq:def_Lkappa}
  L_\kappa B 
  :=  -dB'' - (\i-\zeta)B-2\i\kappa |a|^2B-\i\kappa a^2\bar B
\end{equation}
so that \eqref{eq:bvp_A1A2} may be rewritten as 
\begin{equation}\label{eq:bvp_A1A2_fixedpointformulation}
  L_\kappa A = 2|a|^2A+a^2 \bar A.
\end{equation}   
The fact that $L_\kappa^{-1}:L^2(\T)\to H^1(\T)$ exists as a bounded
linear operator will follow from the injectivity of $L_\kappa$, since $L_\kappa$ is a Fredholm operator of
index $0$. The injectivity is a consequence of the following estimate. For $g\in L^2(\T)$ let $B\in H^2(\T)$
satisfy $L_\kappa B=g$. Testing with $\bar B$ yields 
$$ 
  \int_0^{2\pi} \Big( d |B'|^2 -(\i-\zeta)|B|^2-2\i\kappa |a|^2 |B|^2- \i\kappa a^2 \bar B^2\Big) \,dx =
  \int_0^{2\pi} g \bar B\,dx.
$$
Taking the real and imaginary part of this equation implies 
\begin{align}
d\norm{B'}_2^2+\zeta\norm{B}_2^2+\kappa \Im\int_0^{2\pi} a^2\bar B^2\,dx &= \Re\int_0^{2\pi}g\bar B\,dx,\label{eq:1} \\
\norm{B}_2^2+\kappa\int_0^{2\pi}\Bigl(\underbrace{2|a|^2|B|^2+\Re(a^2 \bar B^2)}_{\geq |a|^2|B|^2}\Bigr)\,dx &= -\Im\int_0^{2\pi}g\bar B\,dx.\label{eq:2}
\end{align}
From \eqref{eq:2} and $\kappa\geq 0$ we get $\|B\|_{2} \leq \|g\|_2$. Together with \eqref{eq:1}, \eqref{eq:2} we obtain
\begin{align*}
  \abs{d}\norm{B'}_2^2+\sign(d)\zeta\norm{B}_2^2 -
  \kappa\int_0^{2\pi}|a|^2|B|^2\,dx&\leq\norm{g}_2^2,\\
  \norm{B}_2^2+\kappa\int_0^{2\pi}|a|^2|B|^2\,dx &\leq\norm{g}_2^2.
\end{align*}
Multiplying the second equation with $\sigma\geq 1$ and summing up both equations we finally get 
\begin{align*}
  \abs{d}\norm{B'}_2^2
  +(\sigma+\sign(d)\zeta)\norm{B}_2^2 \leq (\sigma+1) \norm{g}_2^2. 
\end{align*}
Choosing $\sigma$ sufficiently large and $\gamma$ from~\eqref{eq:def_H1norm} sufficiently small we obtain $\|B\|_{H^1}^2 \leq 4 \|g\|_2^2$. This implies in particular the injectivity of $L_\kappa$, consequently the boundedness of $L_\kappa^{-1}: L^2(\T)\to H^1(\T)$ and finally also the norm bound $\|L_\kappa^{-1}\| \leq  2$ uniformly in $\kappa>0$. 

\medskip

Having proven this bound, we turn to the task to prove that solutions 
$A$ of \eqref{eq:bvp_A1A2} are trivial for $\kappa>\kappa^*$. In view
of~\eqref{eq:bvp_A1A2_fixedpointformulation} we define the bounded linear operator
$$
  K_a B := L_\kappa^{-1}\Big( 2|a|^2B+a^2\bar B\Big): 
  L^2(\T)\to L^2(\T).
$$ 
It remains to show that its operator norm is smaller than 1, because then $K_a$ is a contraction and therefore admits a unique fixed point $A$, which must be the trivial one. Since 
$$
\|2|a|^2B+a^2\bar B\|_2^2 = \int_0^{2\pi} \Big( 5|a|^4 |B|^2+ 2 |a|^2\bar a^2 B^2+2|a|^2a^2\bar B^2\Big)\,dx
\leq 9\|a\|_\infty^4 \|B\|_2^2 $$
we find that
\begin{align*}
\norm{K_a}
  \leq 3\|L_\kappa^{-1}\|\norm{a}_\infty^2   
  \stackrel{\eqref{eq:Linfty_estimate},\eqref{eq:L2}}{\leq} 6\left(1+2\pi^2f^2\abs{d}^{-1}\right)^2
  \Big(\frac{f^2}{\kappa}\Big)^{2/3},
\end{align*}
which is smaller than 1 for $\kappa>\kappa^*$. This finishes the proof. 
\end{proof}

\section{Proof of Theorem~\ref{thm:flow}}

Let us first recall a global existence and uniqueness results in the case $\kappa=0$. It is shown in
Theorem~2.1 in~\cite{jami:14} that \eqref{eq:LLE_timedependent_standard} with $a(0)=\phi\in H^4(\T)$ has a
unique solution $a\in C(\R_+,H^4(\T))\cap C^1(\R_+,H^2(\T))\cap C^2(\R_+,L^2(\T))$. The proof of this result
may be adapted to the case $\kappa>0$ since the crucial estimate (6) in that paper is even better when
$\kappa>0$ given that the damping effect is stronger. The remaining parts of the proof need not be modified
so that we get the same estimates and gobal well-posedness result as in~\cite{jami:14} also in the case
$\kappa>0$.
Since we will need the inequality $\|a(t)\|_2\leq \max\{\sqrt{2\pi}|f|,\|a(0)\|_2\}$ in the proof of our
convergence results, let us prove this first.  For notational convenience we suppress the spatial variable in
our notation.

\medskip

For any given solution $a$ of~\eqref{eq:LLE_timedependent} the following estimate holds
  \begin{align*}
    \frac{d}{dt} \left(\frac{\|a(t)\|_2^2}{2}\right)
    =&  \Re\left( \int_0^{2\pi} a_t(t)\ov{a(t)}\,dx \right) \\
    \stackrel{\eqref{eq:LLE_timedependent}}{=}&  \Real\left( \int_0^{2\pi} 
   \Big( (- 1-\i\zeta+(\i-\kappa)|a(t)|^2)a(t)+ f +\i da_{xx}(t)\Big)  \ov{a(t)} \,dx \right)\\
   =& - \|a(t)\|_2^2 - \kappa \|a(t)\|_4^4 + f  \int_0^{2\pi} \Real(a(t))\,dx \\
   \leq& - \|a(t)\|_2^2-\frac{\kappa}{2\pi}\|a(t)\|_2^4 +\sqrt{2\pi}|f|\|a(t)\|_2.
  \end{align*}
  So $\|a(t)\|_2$ decreases provided the last term is negative. Since this is true is precisely for
  $\|a(t)\|_2\geq \sqrt{2\pi}\tilde C_\kappa$ by~\eqref{basis_fuer_cardano},\eqref{eq:L2},  we conclude
  \begin{equation}\label{eq:L2_est}
    \|a(t)\|_2 
    \leq \max\{ \sqrt{2\pi}\tilde C_\kappa, \|a(0)\|_2\}
    \stackrel{\eqref{eq:L2}}{\leq} \max\{ \sqrt{2\pi}|f|, \|a(0)\|_2\} 
    \qquad\text{for all }t\geq 0.
  \end{equation}

\medskip  

Furthermore, using the equation for $a$ and integration by parts we get 
\begin{align*}
  \frac{d}{dt} \left( \frac{\|a_x(t)\|_2^2}{2}\right)  
  =& \Re \left(\int_0^{2\pi}  a_{xt}(t) \ov{a_x(t)} \,dx \right)\\
  =& - \Re\left(\int_0^{2\pi}  a_t(t) \ov{a_{xx}(t)} \,dx \right)\\
  \stackrel{\eqref{eq:LLE_timedependent}}{=}& - \Real\left( \int_0^{2\pi} 
  \Big( (- 1-\i\zeta+(\i-\kappa)|a(t)|^2)a(t)+ f +\i da_{xx}(t)\Big)  \ov{a_{xx}(t)} \,dx \right)\\  
  =& - \int_0^{2\pi} |a_x(t)|^2 \,dx - \kappa \int_0^{2\pi}|a(t)|^2|a_x(t)|^2 \,dx - 2\kappa
  \int_0^{2\pi} \Re\big(a(t) \ov{a_x(t)}\big)^2 \,dx \\
  &  - 2\int_0^{2\pi} \Im\big(a(t)\ov{a_x(t)}\big)\Re\big(a(t)\ov{a_x(t)}\big) \,dx.
\end{align*}
Writing $a\bar a_x=s+\i r$ and using the scalar inequality
\begin{equation}\label{eq:flow_scalar_ineq}
   - \kappa (s^2+r^2) - 2\kappa s^2 - 2sr 
   \leq \underbrace{(-2\kappa+\sqrt{1+\kappa^2})}_{=:\alpha_\kappa}(s^2+r^2) \qquad (s,r\in\R)
\end{equation}
we  get the estimate 
$$
  \frac{d}{dt} \left(\frac{\|a_x(t)\|_2^2}{2}\right)
  \leq - \|a_x(t)\|_2^2 +  \alpha_\kappa  \int_0^{2\pi}
  |a(t)|^2|a_x(t)|^2 \,dx
  \qquad\text{for all }t\geq 0. 
$$
Since we assumed $\kappa\geq \frac{1}{\sqrt 3}$, we have $\alpha_\kappa\leq 0$ so that
$\|a_x(t)\|_2^2$ decays exponentially to 0. The Poincar\'{e}-Wirtinger inequality implies
$\|a(t)-\frac{1}{2\pi}\int_0^{2\pi} a(t)\,dx \|_2$ decays exponentially as $t\to \infty$. 
The $L^2$-boundedness of $a(t)$ derived in~\eqref{eq:L2_est} now implies that the sequence
$\int_0^{2\pi} a(t) \,dx$ is bounded, hence $a(t_m)$ converges in $L^2(\T)$ for some sequence $t_m\nearrow\infty$ to some
constant solution $a^*$ of \eqref{eq:LL_complex}. It remains to prove that this actually implies the
convergence of the whole sequence.

\medskip 

By the fundamental theorem of calculus we get 
\begin{equation}\label{eq:aux_est}
  \|a(t)-a^*\|_\infty
  \leq \|a_x(t)\|_1 + \min_{[0,2\pi]} |a(t)-a^*| 
  \leq \sqrt{2\pi}\|a_x(t)\|_2 + \frac{1}{\sqrt{2\pi}} \|a(t)-a^*\|_2. 
\end{equation}
In particular, the subsequence $a(t_m)$ converges uniformly to the constant $a^*$. 
 So for any given $\delta\in (0,1)$ we can find an $\eps>0$ such that all $h\in\C$ with
$|h|<\eps$ satisfy the inequality
\begin{align}\label{eq:flow_scalarineq}
  \begin{aligned}
  & \Re \left( (\i-\kappa)\Big( |a^*+h|^2(a^*+h)- |a^*|^2a^*\Big)\ov{h}\right) \\
  &= -\kappa|a^*|^2|h|^2 - 2\kappa \Bigl(\Re \big(a^*\ov h\big)\Bigr)^2   -2 \Im\big(a^* \ov h) \Re(a^*\ov
  h) + O(|h|^3)  \\
  &\stackrel{\eqref{eq:flow_scalar_ineq}}{\leq } \alpha_\kappa |a^*|^2|h|^2  + O(|h|^3) \\
  &\leq \delta |h|^2.
\end{aligned}
\end{align}
Here we used $\alpha_\kappa\leq 0$. Choosing $t_m$ large enough we
can achieve 
\begin{equation}\label{eq:choice_tm}
  \|a(t_m)-a^*\|_2  \leq \frac{\sqrt{2\pi}}{4}\eps\qquad\text{and}\qquad  
  \|a_x(t)\|_2 \leq \frac{1}{4\sqrt{2\pi} }\eps \quad\text{for all }t\geq t_m.
\end{equation}
So the function $h(t):= a(t)-a^*$ satisfies for $t\geq t_m$  
the following  differential inequality provided $\|h(t)\|_\infty\leq \eps$
\begin{align*}
  \frac{d}{dt} \left(\frac{\|h(t)\|_2^2}{2}  \right)
  &= \Re\left( \int_0^{2\pi}\partial_t h(t) \ov{h(t)}  \,dx\right)  \\
  &\stackrel{\eqref{eq:LLE_timedependent}}{=}
    - \|h(t)\|_2^2  +  \Re\left( (\i-\kappa)
    \int_0^{2\pi} \big( |a^*+h(t)|^2(a^*+h(t))-   |a^*|^2a^* \big)\ov{h(t)}  \,dx \right) \\
  &\stackrel{\eqref{eq:flow_scalarineq}}{\leq} (-1+\delta) \|h(t)\|_2^2.
\end{align*}
Given that $\|h(t_m)\|_\infty\leq \frac{\sqrt{2\pi}}{4}\eps<\eps$ we infer that
$\|h(t)\|_2=\|a(t)-a^*\|_2$ decreases on some maximal interval $(t_m,t_m+T)$ and we want to show
$T=\infty$.
From \eqref{eq:choice_tm} we infer $$
 \|h(t)\|_2  \leq \frac{\sqrt{2\pi}}{4}\eps,\quad 
  \|h_x(t)\|_2 \leq \frac{1}{4\sqrt{2\pi}}\eps \qquad\text{for all }t\in [t_m,t_m+T]
$$  
so that~\eqref{eq:aux_est} implies 
$$
  \|h(t)\|_\infty 
  \leq \sqrt{2\pi}\cdot \frac{1}{4\sqrt{2\pi} }\eps + \frac{1}{\sqrt{2\pi}}\cdot
  \frac{\sqrt{2\pi}}{4}\eps
  \leq \frac{\eps}{2}<\eps
  \qquad\text{for all }t\in [t_m,t_m+T].
$$
As shown above, this implies that $\|h(t)\|_2$ is decreasing on a right neighbourhood of $t_m+T$. So we
conclude that there cannot be a finite maximal $T$ with the property mentioned above. As a consequence, $T=\infty$, $\|h(t)\|_2$
is decreasing on $[t_m,\infty)$ and we obtain $\|a(t)-a^*\|_{H^1(\T)}=\|h(t)\|_{H^1(\T)}\to 0$ as
claimed. This finishes the proof.\qed

\medskip

We add an extension of this result that covers damping parameters $\kappa<\frac{1}{\sqrt 3}$. In this case we
may obtain the convergence of the flow provided the initial condition $\phi=a(0)$ has the property that  
$\|\phi_x\|_2$ and $\|\phi\|_2$ are not too large.

\begin{lem}\label{lem:flow}
  Assume $d\neq 0, \zeta,f\in\R$ and $\kappa<\frac{1}{\sqrt 3}$. Assume
  that the initial condition  $a(0)=\phi\in H^4(\T)$  satisfies
  \begin{equation}\label{eq:flow_initcondition_assumption}
    2\pi \|\phi_x\|_2 + \max\{\sqrt{2\pi} |f|,\|\phi\|_2\} <\sqrt{\frac{2\pi}{\alpha_\kappa}}. 
  \end{equation}
  Then the uniquely determined solution $a\in C(\R_+,H^4(\T))$ of~\eqref{eq:LLE_timedependent} converges in $H^1(\T)$ to a constant. 
\end{lem}
\begin{proof}
  We argue as above. Using the same estimate as in the above proof we get now for $\alpha_\kappa>0$
  \begin{align*}
  \frac{d}{dt} \left(\frac{\|a_x(t)\|_2^2}{2}\right)
  &\leq - \|a_x(t)\|_2^2  + \alpha_\kappa \int_0^{2\pi} |a(t)|^2|a_x(t)|^2 \\
  &\leq (-1+\alpha_\kappa\|a(t)\|_\infty^2) \|a_x(t)\|_2^2\\
  &\stackrel{\eqref{eq:aux_est}}{\leq}  \left(-1+\alpha_\kappa
  (\sqrt{2\pi}\|a_x(t)\|_2+\frac{1}{\sqrt{2\pi}}\|a(t)\|_2)^2\right) \|a_x(t)\|_2^2  \\
  &\stackrel{\eqref{eq:L2_est}}{\leq} \left(-1+\alpha_\kappa
  (\sqrt{2\pi}\|a_x(t)\|_2+\frac{1}{\sqrt{2\pi}} \max\{\sqrt{2\pi}|f|,\|a(0)\|_2\})^2\right)
  \|a_x(t)\|_2^2.
  \end{align*}
  So the prefactor is negative for small $t>0$ by assumption \eqref{eq:flow_initcondition_assumption}.
  Hence, by monotonicity, it remains negative for all $t>0$ and we conclude as above. 
\end{proof}

We do not know whether the above convergence result is sharp in the sense that there are initial data
causing non-convergence or even blow-up in infinite time. As above we moreover infer that all nonconstant
stationary solutions $a$ for $\kappa<\frac{1}{\sqrt 3}$ satisfy 
$$
    2\pi \|a_x\|_2 + \max\{\sqrt{2\pi} |f|,\|a\|_2\} \geq  \sqrt{\frac{2\pi}{\alpha_\kappa}}.
$$


\section{Proof of Theorem~\ref{thm:continuation}} \label{sec:continuation}

In this section we discuss \eqref{eq:LL_realline} in the case of anomalous dispersion $d>0$, and we will
prove the existence of solitary-type localized solutions. At the end of this section we explain why our method fails in the case of normal dispersion $d<0$.

\medskip

Let us consider a rescaled version of~\eqref{eq:LL_realline} given by 
\begin{equation} \label{eq:onR}
  -du''+(\tilde{\zeta}-\varepsilon \mathrm i)u-(1+\mathrm i\kappa)\abs{u}^2u+\mathrm i\tilde{f}=0 \quad
  \text{on }\mathbb{R}, \qquad u'(0)=0
\end{equation}
for $d,\tilde{\zeta}>0$ and $\varepsilon,\kappa\geq 0$.
Notice that $u$ solves~\eqref{eq:onR} with $\tilde{\zeta},\tilde f$
  if and only if $a(x):=\varepsilon^{-1/2}u(\varepsilon^{-1/2}x)$ solves~\eqref{eq:LL_complex} with
 $\zeta=\tilde{\zeta}\varepsilon^{-1}$ and $f=\tilde{f}\varepsilon^{-3/2}$ on $\mathbb{R}$.

\medskip

We consider solutions of \eqref{eq:onR} of the form $u=\tilde u+u^\infty$, where 
$\tilde u$ belongs to the space $H^2_{even}(\RR;\CC)$ of even complex-valued $H^2$-functions on the real
line and $u^\infty:=\lim_{|x|\to\infty} u(x)$ solves the algebraic equation
\begin{align}\label{eq:alg}
(\tilde{\zeta}-\varepsilon \mathrm i)u-(1+\mathrm i\kappa)\abs{u}^2u+\mathrm i\tilde{f}=0.
\end{align}
The strategy is to find two purely imaginary solutions of \eqref{eq:onR} in the special case
$\varepsilon=\kappa=0$ and to continue them into the situation $\varepsilon,\kappa>0$ via the implicit
function theorem. More precisely, Theorem~\ref{thm:continuation} is proved once we have
shown Theorem~\ref{thm:continuation_of_homoclinics} below.

\medskip

Let us begin with the case $\epsilon=\kappa =0$, where we consider solutions of
\begin{align} \label{eq:onR_zero}
  -du''+\tilde{\zeta} u-\abs{u}^2u+\mathrm i\tilde{f}=0 \quad
  \text{on }\mathbb{R},\qquad u'(0)=0
\end{align}
and where $u^\infty=\lim_{|x|\to \infty} u(x)$ satisfies 
\begin{align}\label{eq:alg_zero}
\tilde{\zeta}u-\abs{u}^2u+\mathrm i\tilde{f}=0.
\end{align} 
We will always work in the setting where \eqref{eq:alg_zero} has three distinct solutions. Let us briefly explain why this is fulfilled for 
$0\leq|\tilde{f}|<\frac{2\sqrt{3}}{9}\tilde{\zeta}^{3/2}$. Clearly, \eqref{eq:alg_zero} only has purely
imaginary solutions $u^\infty = \i v^\infty$, where $v^\infty\in \R$ solves
\begin{align}\label{eq:alg1}
-\tilde{\zeta}v+v^3=\tilde{f}.
\end{align} 
The function $v \mapsto -\tilde{\zeta}v+v^3$ has the local minimum $-\frac{2\sqrt{3}}{9}\tilde{\zeta}^{3/2}$ and
the local maximum $\frac{2\sqrt{3}}{9}\tilde{\zeta}^{3/2}$. Therefore, if
$0\leq |\tilde{f}|<\frac{2\sqrt{3}}{9}\tilde{\zeta}^{3/2}$ then there are three distinct solutions $v^{(j)}$,
$j=1,2,3$ of \eqref{eq:alg1} with $v^{(1)}<-\frac{\sqrt{\zeta}}{\sqrt{3}}<v^{(2)}<\frac{\sqrt{\zeta}}{\sqrt{3}}<v^{(3)}$.

\medskip

Next let us discuss the existence of two homoclinic solutions of \eqref{eq:onR_zero}. Their nondegeneracy will be proved in Proposition~\ref{prop:nondeg}.

\begin{prop}\label{prop:continuation}
Let $d,\tilde{\zeta}>0$ and $0\leq |\tilde{f}|<\frac{2\sqrt{3}}{9}\tilde{\zeta}^{3/2}$. 
There exist two purely imaginary and even solutions $u_i=\tilde{u}_i+u_i^\infty$
 of~\eqref{eq:onR_zero} with $\tilde{u}_i\in H^2_{even}(\RR;\CC)$  for $i=1,2$  and $x\Im(u_1')>0$ and $x\Im(u_2')<0$ on $\mathbb{R}\backslash\{0\}$.
\end{prop}

\begin{proof}
Looking for purely imaginary even homoclinic solutions $u=\i v$ of \eqref{eq:onR_zero} means that we need to find a real-valued even homoclinic solution $v$ of  
\begin{align}\label{eq:transf}
-dv''+\tilde{\zeta}v-v^3+\tilde{f}=0 \quad \text{on }\R.
\end{align}
The corresponding first integral is given by
\begin{align*}
I(v',v):=-dv'^2 + \tilde{\zeta}v^2-\frac{1}{2}v^4+2\tilde{f}v.
\end{align*}
All trajectories of \eqref{eq:transf} are therefore bounded in the $(v,v')$-plane and symmetric with respect
to the $v$-axis. Moreover, every trajectory crosses the $v$-axis.

\medskip

The equilibria of \eqref{eq:transf} are given by the solutions of the algebraic equation~\eqref{eq:alg1}. As we have seen, there are three
distinct real-valued solutions $v^{(j)}, j=1,2,3$ for
$0\leq |\tilde{f}|<\frac{2\sqrt{3}}{9}\tilde{\zeta}^{3/2}$.
The eigenvalues of the linearization in $v^{(j)}$ satisfy
\begin{equation*}
\lambda_{1,2}^{(j)}=\pm\sqrt{-\Delta^{(j)}}/\sqrt{d} \qquad\text{with }\Delta^{(j)}:=-\tilde{\zeta}+3({v^{(j)}})^2.
\end{equation*}
The linear stability analysis, which allows us to characterize the equilibria of the nonlinear system, reduces to the analysis of $\Delta^{(j)}$. Observe that we have $-\tilde{\zeta}+3v^2<0$ on $(-\frac{\sqrt{\zeta}}{\sqrt{3}},\frac{\sqrt{\zeta}}{\sqrt{3}})$ and $-\tilde{\zeta}+3v^2>0$ on $\RR\backslash[-\frac{\sqrt{\zeta}}{\sqrt{3}},\frac{\sqrt{\zeta}}{\sqrt{3}}]$. Hence, for $v^{(1)}<v^{(2)}<v^{(3)}$, we have $\Delta^{(j)}>0$ for $j=1,3$ and $\Delta^{(j)}<0$ for $j=2$. This means that for $j=2$ we have
two real eigenvalues of opposite sign, and the equilibrium is an unstable saddle. For $j=1,3$, we have purely imaginary eigenvalues of opposite sign, and hence, these equilibria are stable centers surrounded by periodic orbits. 

\medskip

Since the unstable manifold of the saddle is symmetric around the $v$-axis it connects to the stable manifold
and thus provides the two homoclinic orbits.
\end{proof}

\begin{figure}
\centering
\begin{tikzpicture}
  \draw[->=.4] (-4,0) to (4,0) node[below] {$v$};
  \draw[->=.6] (0,-2) to (0,2) node[right] {$v'$}; 
  \draw[-<-=.5] (-3,0) to [out=-90,in=-140] (-0.6,0);
	\draw[->-=.5] (-3,0) to [out=90,in=140] (-0.6,0);
  \draw[-<-=.5] (-0.6,0) to [out=-40,in=-90] (3,0);
	\draw[->-=.5] (-0.6,0) to [out=40,in=90] (3,0);
 \end{tikzpicture}
\caption{Homoclinic orbits for $\varepsilon=\kappa=0$}
\end{figure}
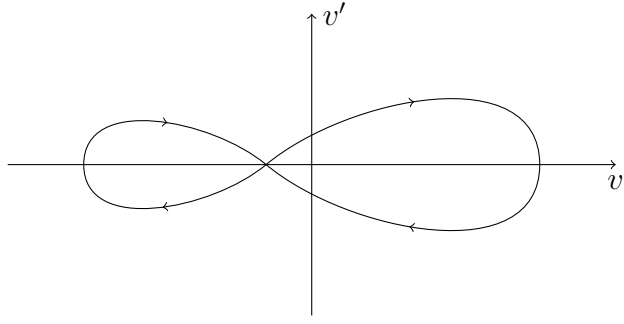

For the following nondegeneracy result let us recall from Section~\ref{sec:bifurcation}  the notation
$g(u)=|u|^2u-\i f$, $Dg(u)z = 2|u|^2 z+ u^2\bar z$ for $u,z\in \C$. 

\begin{prop}\label{prop:nondeg} 
Let $d,\tilde{\zeta}>0$ with $0< |\tilde{f}|<\frac{2\sqrt{3}}{9}\tilde{\zeta}^{3/2}$. If $u_1, u_2$ are the two homoclinic solutions 
of \eqref{eq:onR_zero} then 
$$
\operatorname{ker}_{H^2(\mathbb{R};\CC)}\Bigl(-d\frac{d^2}{dx^2}+\tilde{\zeta}-Dg(u_i)\Bigr)
  = \operatorname{span} \{u_i'\}
$$
for $i=1,2$.  
\end{prop}

\begin{remark} \label{remark_on_nondeg} Here $\operatorname{ker}_{H^2(\mathbb{R};\CC)}(-d\frac{d^2}{dx^2}+\tilde{\zeta}-Dg(u_i))$ refers to the kernel of the differential operator on the domain $H^2(\mathbb{R};\CC)$.
As a consequence, if we set the domain of the differential operator as
$H^2_{even}(\mathbb{R};\CC)$ we get 
$\operatorname{ker}_{H^2_{even}(\RR;\CC)}(-d\frac{d^2}{dx^2}+\tilde{\zeta}-Dg(u_i)) = \{0\}$ for $i=1,2$. This is true because
$H^2_{even}(\RR;\CC)$ only contains functions with $u'(0)=0$ so that $u_i'\not \in H^2_{even}(\RR;\CC)$ for
$i=1,2$ because \eqref{eq:onR_zero} yields
$$ 
  du_i''(0)=\tilde{\zeta}u_i(0)-|u_i(0)|^2u_i(0)+\i\tilde{f}.
$$
 The latter expression is non-zero since $u_i(0)\neq \i v^{(j)}$ for $j=1,2,3$ and $\i v^{(1)}, \i v^{(2)}, \i
 v^{(3)}$ are the solutions of \eqref{eq:alg1}. Note also that the proposition applies only for $\tilde
f\not =0$ because for $\tilde f=0$ scaling with a complex phase factor produces another degeneracy so that $\spann\{u_i', \i u_i\}\subset \operatorname{ker}_{H^2(\mathbb{R};\CC)}(-d\frac{d^2}{dx^2}+\tilde{\zeta}-Dg(u_i))$ and this time $\i u_i\in H^2_{even}(\RR;\CC)$. 
\end{remark}

\begin{proof}
We prove nondegeneracy only for $i=1$. Since $u_1'\in H^2(\mathbb{R};\CC)$, we may differentiate \eqref{eq:onR_zero} to see that
$u_1'\in\operatorname{ker}_{H^2(\mathbb{R};\CC)}\Bigl(-d\frac{d^2}{dx^2}+\tilde{\zeta}-Dg(u_1)\Bigr)$.

\medskip 

For the converse inclusion, let
$\varphi=\varphi_1+i\varphi_2 \in\operatorname{ker}_{H^2(\mathbb{R};\CC)}\Bigl(-d\frac{d^2}{dx^2}+\tilde{\zeta}-Dg(u_1)\Bigr)$ 
for real-valued $\varphi_1,\varphi_2$. Then we have 
\begin{align}
\Bigl(-d\frac{d^2}{dx^2}+\tilde{\zeta}-v_1^2\Bigr)\varphi_1&=0 \quad\mbox{on } \R, \label{eq:sys1}\\
\Bigl(-d\frac{d^2}{dx^2}+\tilde{\zeta}-3v_1^2\Bigr)\varphi_2&=0 \quad\mbox{on } \R, \label{eq:sys2}
\end{align}
and we need to show that $\varphi_1=0$ and that $\varphi_2$ is a real multiple of $v_1'$. Due to \eqref{eq:transf} we also see that 
\begin{equation} \label{eq:for_v1strich}
-d v_1''' + (\tilde\zeta-3v_1^2)v_1'=0 \mbox{ on } \R \quad \mbox{ and } \quad\int_0^\infty d(v_1'')^2 -(\tilde\zeta-3v_1^2)(v_1')^2\,dx =0.
\end{equation}
 We split $\varphi_1:=\varphi_{1,even}+\varphi_{1,odd}$ into even and odd part. Then we observe that
$\varphi_{1,even}\in H^2(\RR;\RR)$ solves~\eqref{eq:sys1} with $\varphi_{1,even}'(0)=0$ and that
$\varphi_{1,odd}\in H^2(\RR;\RR)$ solves~\eqref{eq:sys1} with $\varphi_{1,odd}(0)=0$.

\medskip

Let us first show that either $\varphi_{1,even}\equiv 0$ or $\varphi_{1,even}$  has no zero on $\R$. 
Indeed, if $\varphi_{1,even}$ had a first positive zero $x_0>0$ with $\varphi_{1,even}(x_0)=0$ then w.l.o.g.
$\varphi_{1,even}>0$ on $(0,x_0)$. Since $\varphi_{1,even}(x)\to 0$ as $x\to \infty$ there is $x_1\in
(x_0,\infty]$ such that $\varphi_{1,even}<0$ on $(x_0,x_1)$ and $\lim_{x\to x_1} \varphi_{1,even}(x)=0$. If
we multiply the differential equation in \eqref{eq:for_v1strich} by $\varphi_{1,even}$ and subtract
\eqref{eq:sys1} multiplied by $v_1'$ then we find
\begin{align*}
0&=\int_{x_0}^{x_1}-d(v_1'''\varphi_{1,even}-\varphi_{1,even}''v_1')\,dx-\int_{x_0}^{x_1} 2v_1^2\underbrace{v_1'}_{>0}\underbrace{\varphi_{1,even}}_{<0}\,dx \\
&\geq-\int_{x_0}^{x_1} d\frac{d}{dx}(v_1''\varphi_{1,even}-\varphi_{1,even}'v_1')\,dx\\
&=-d\paren{\underbrace{\varphi_{1,even}'(x_0)}_{<0}\underbrace{v_1'(x_0)}_{>0}-\underbrace{\varphi_{1,even}'(x_1)}_{\geq 0}\underbrace{v_1'(x_1)}_{\geq 0}}>0.
\end{align*}
This is impossible and proves the assertion that $\varphi_{1,even}$ has no zero on $\R$. An almost identical
argument applied to $\varphi_{1,odd}$  provides the alternative $\varphi_{1,odd}\equiv 0$ or $\varphi_{1,odd}$ has no zero on $(0,\infty)$.

\medskip

Now suppose $\varphi_{1,odd}\not\equiv 0$. Then $\varphi_{1,odd}\in H^1_0((0,\infty);\R)$
is w.l.o.g a positive Dirichlet eigenfunction to the eigenvalue $0$ of
$L_1:=-d\frac{d^2}{dx^2}+\zeta-v_1^2$ on $(0,\infty)$.
Observe that  $v_1'\in H^1_0((0,\infty);\R)$ is a positive Dirichlet eigenfunction of $L_2:=-d\frac{d^2}{dx^2}+\zeta-3v_1^2$ on $(0,\infty)$
corresponding to the smallest eigenvalue $0$. We also have the following inequality between the quadratic forms of $L_2$ and $L_1$
\begin{align*}
\int_0^\infty d(\phi')^2+(\zeta-3v_1)^2\phi^2\,dx<\int_0^\infty d(\phi')^2+(\zeta-v_1)^2\phi^2\,dx
\end{align*}
for all $\phi\in H^1_0((0,\infty);\R)\setminus\{0\}$. 
Therefore, by Poincar\'{e}'s min-max principle we obtain a strict ordering between the Dirichlet eigenvalues
of $L_1$ and $L_2$, and hence the smallest Dirichlet eigenvalue of $L_1$ is strictly positive, which yields a contradiction. This
implies  $\varphi_{1,odd}\equiv 0$.

\medskip

Let us now consider the even part $\varphi_{1,even}$ and suppose that $\varphi_{1,even}$ has no zero on $(0,\infty)$. 
Testing \eqref{eq:transf} with $\varphi_{1,even}$, integrating twice we obtain
\begin{align*}
0&\neq-\int_0^\infty \tilde{f}\varphi_{1,even}\,dx=\int_0^\infty -dv_1''\varphi_{1,even}+(\zeta-v_1^2)v_1\varphi_{1,even}\,dx\\
&=\int_0^\infty -dv_1\varphi_{1,even}''+(\zeta-v_1^2)v_1\varphi_{1,even}\,dx=0.
\end{align*}
This is a contraction, and hence, $\varphi_{1,even}\equiv 0$. Together with $\varphi_{1,odd}\equiv 0$ we finally see $\varphi_1\equiv 0$.

\medskip

Now we show that $\varphi_2$ is a multiple of $v_1'$. Multiplying the differential equation in \eqref{eq:for_v1strich} with $\varphi_2$, 
\eqref{eq:sys2} with $v_1'$ and subtracting we obtain
\begin{align*}
0=-d(v_1'''\varphi_2-\varphi_2''v_1')=-d\frac{d}{dx}(v_1''\varphi_2-\varphi_2'v_1').
\end{align*}
As both $v_1'$ and $\varphi_2$ together with their derivatives vanish at infinity we obtain
\begin{align*}
v_1''\varphi_2-\varphi_2'v_1'=0,
\end{align*}
which means that $\varphi_2$ is a multiple of $v_1'$. This concludes the proof of the proposition.
\end{proof}

Now we will continue the purely imaginary nontrivial solutions $u_1, u_2$ of \eqref{eq:onR_zero} from Proposition~\ref{prop:nondeg} into the range 
where $\varepsilon,\kappa>0$.  For the proof of the final result, we rewrite~\eqref{eq:onR} for $u=\tilde u
+u^\infty_{\epsilon,\kappa}$ with $\tilde u\in H$ as follows
\begin{align} \label{eq:rewrite}
-d\tilde{u}''+\tilde{\zeta}\tilde{u}-\varepsilon\mathrm i\tilde{u}- (1+\mathrm i
\kappa)(g(\tilde{u}+u_{\varepsilon,\kappa}^\infty)-g(u_{\varepsilon,\kappa}^\infty))=0.
\end{align}
Here $u^\infty_{\varepsilon,\kappa}$ is given as the continuation of $\i v^{(2)}$ into the range of
$\varepsilon, \kappa>0$. Note that three distinct solutions of~\eqref{eq:alg} persists for small
$\varepsilon,\kappa>0$.
%
%
%

\begin{thm} \label{thm:continuation_of_homoclinics}
  Let $u_1=\tilde u_1+\i v^{(2)}, u_2=\tilde u_2+\i v^{(2)}$ be the two solutions of \eqref{eq:onR} for
  $(\varepsilon,\kappa)=(0,0)$ from Proposition~\ref{prop:continuation}. Then there exist open neighborhoods
  $U_i$ of $\tilde{u}_i$ in  $H^2_{even}(\RR;\CC)$, $J_i$ of $(0,0)$ in $\mathbb{R}\times\mathbb{R}$ such
  that~\eqref{eq:rewrite} is uniquely solvable for $(\tilde u,\epsilon, \kappa)$ in $U_i\times J_i$, $i=1,2$.
\end{thm}

\begin{proof}
We define $F:H^2_{even}(\RR;\CC)\times\RR\times\RR\to L^2_{even}(\R;\C)=\{\tilde u\in L^2(\R;\C): \tilde u(-x)=\tilde u(x) \text{ for a.a. } x\in\R\}$ by
\begin{align*}
F(\tilde{u},\varepsilon,\kappa):=-d\tilde{u}''+\tilde{\zeta}\tilde{u}-\varepsilon\mathrm i\tilde{u}-(1+\mathrm i \kappa)(g(\tilde{u}+u_{\varepsilon,\kappa}^\infty)-g(u_{\varepsilon,\kappa}^\infty)).
\end{align*}
Then we have $F(\tilde{u}_i,0,0)=0$ by definition of $\tilde{u}_i$ and $\frac{\partial F}{\partial\tilde
u}(\tilde u_i,0,0) =-d\frac{d^2}{dx^2}+\tilde{\zeta}-Dg(u_i): H^2_{even}(\RR;\CC)\to L^2_{even}(\R;\C)$. Due
to Remark~\ref{remark_on_nondeg} we know that $\operatorname{ker}_{H^2_{even}(\RR;\CC)}(\frac{\partial
F}{\partial\tilde u}(\tilde u_i,0,0))=\{0\}$ for $i=1,2$. Since $\frac{\partial F}{\partial\tilde u}(\tilde
u_i,0,0)$ is a Fredholm operator of index $0$, it has a bounded inverse and thus the statement of the theorem
follows from the implicit function theorem.
\end{proof} 

\begin{remark}
Let us denote one of the two solution families of  Theorem~\ref{thm:continuation_of_homoclinics} by
$u_{\varepsilon,\kappa}$. Taking into account the  rescaling
$a_{\varepsilon,\kappa}(x)=\varepsilon^{-1/2}u_{\varepsilon,\kappa}(\varepsilon^{-1/2}x)$ we  have proved
Theorem~\ref{thm:continuation}. Moreover,
\begin{align*}
\Big\Vert a_{\varepsilon,\kappa}-\lim_{\abs{x}\to\infty} a_{\varepsilon,\kappa}(x)\Big\Vert_{H^2}
\geq\varepsilon^{-1/4}\Big\Vert
u_{\varepsilon,\kappa}-\lim_{\abs{x}\to\infty}u_{\varepsilon,\kappa}(x)\Big\Vert_{H^2}\to \infty
\end{align*}
for $\varepsilon\to 0$ uniformly with respect to $\kappa$. 
\end{remark}

We finish our discussion with a brief analysis of the case $d<0$ (normal dispersion). Here, we also consider the rescaled equation \eqref{eq:onR} and write it in the form
\begin{equation} \label{eq:onR_dark}
  -|d|u''+(\varepsilon \mathrm i-\tilde{\zeta})u+(1+\mathrm i\kappa)\abs{u}^2u-\mathrm i\tilde{f}=0 \quad
  \text{on }\mathbb{R}, \qquad u'(0)=0.
\end{equation}
Starting with $\varepsilon=\kappa=0$ we consider purely imaginary solutions. The equilibria in the phase plane for \eqref{eq:transf} are the same as before, but due to $d<0$ their character changes. 
The eigenvalues of the linearization are now given by 
$$
\lambda_{1,2}^{(j)}=\pm\i\sqrt{-\Delta^{(j)}}/\sqrt{|d|}\quad\text{with }\Delta^{(j)}:=-\tilde{\zeta}+3({v^{(j)}})^2
$$
for $j=1,2,3$. Now we have a center for $j=2$ and two unstable saddles for $j=1,3$. The unstable saddles are connected by two heteroclinic solutions. Going back to \eqref{eq:onR_dark} we have for $\varepsilon=0$ two heteroclinic solutions $u_1, u_2$ with  $\Im(u_1')>0$ and $\Im(u_2')<0$ on $\R$. Moreover $\lim_{x\to \infty} u_1(x) = \lim_{x\to-\infty} u_2(x)= u^{(3)}=\i v^{(3)}$, $\lim_{x\to-\infty} u_1(x)=\lim_{x\to\infty} u_2(x)=u^{(1)}=\i v^{(1)}$. For $\varepsilon, \kappa>0$ the unstable saddles persist and one might try to continue the heteroclinic solutions $u_1, u_2$ into the range $\varepsilon, \kappa>0$. Let us explain why the previous continuation argument fails in the case of $u_1$ (the argument for $u_2$ is the same). One could seek for heteroclinic solutions of the form 
$$
u= \tilde u + \psi_{\varepsilon,\kappa} \;\mbox{ with } \tilde u \in H^2(\R)
$$
and where $\psi_{\varepsilon,\kappa}$ is a smooth given function of $x$, continuous in $\varepsilon, \kappa$ with 
$$
\psi_{0,0}=u_1 \quad\mbox{and}\quad \lim_{x\to \infty} \psi_{\varepsilon,\kappa}(x) =
u^{(3)}_{\epsilon,\kappa},\;
\lim_{x\to-\infty} \psi_{\varepsilon,\kappa}(x) =u^{(1)}_{\epsilon,\kappa} 
$$
where $u^{(j)}_{\varepsilon,\kappa}$ are the continuations of the purely imaginary zeros $u^{(j)}$ of \eqref{eq:alg} into the range $\varepsilon,\kappa>0$. The implicit function continuation argument applied to 
$$
F(\epsilon,\kappa,\tilde u) = -|d| (\tilde u+\psi_{\varepsilon,\kappa})'' +\epsilon\i(\tilde u+\psi_{\varepsilon,\kappa})- g( \tilde u+\psi_{\varepsilon,\kappa})
$$
would then provide $\tilde u$ as a function of $\varepsilon$ and $\kappa$. Due to $\psi_{0,0}=u_1$ we have $F(0,0,0)=0$ and 
the linearized operator is given by $\frac{\partial F}{\partial \tilde u}(0,0,0)=-|d|\frac{d^2}{dx^2}-\tilde{\zeta}+Dg(u_1): H^2(\R)\to L^2(\R)$. Now there is the question of nondegeneracy of $u_1$. Since $u_1$ is purely imaginary, $\frac{\partial F}{\partial \tilde u}(0,0,0)$ decouples into two real-valued, selfadjoint operators
\begin{align}
L_1 := &\Bigl(-|d|\frac{d^2}{dx^2}-\tilde{\zeta}+v_1^2(x)\Bigr): H^2(\R)\to L^2(\R), \label{eq:sys1_dark}\\
L_2 := &\Bigl(-|d|\frac{d^2}{dx^2}-\tilde{\zeta}+3v_1^2(x)\Bigr): H^2(\R)\to L^2(\R). \label{eq:sys2_dark}
\end{align}
Since $v^{(j)}$ solves $(-\tilde\zeta +v^2)v=f>0$ for $j=1,2,3$ and $v^{(1)}<0<v^{(3)}$ we see that $-\tilde\zeta +\lim_{x\to-\infty} v_1^2(x) = -\tilde\zeta + (v^{(1)})^2<0$ and $-\tilde\zeta +\lim_{x\to-\infty} v_1^2(x) = -\tilde\zeta + (v^{(3)})^2>0$. Hence we get for the essential spectrum of $L_1$ the relation
$$
\sigma_{ess}(L_1) = [-\tilde\zeta + (v^{(1)})^2,\infty)
$$
and $0\in \sigma_{ess}(L_1)$. Unlike in the case of $d>0$ , $L_1$ is not a Fredholm operator and the non-degeneracy of the heteroclinic solution fails for $d<0$.

\bibliographystyle{plain}
\bibliography{bibliography}

\begin{thebibliography}{10}

\bibitem{chembo_2010}
Y.~K. Chembo and Nan Yu.
\newblock Modal expansion approach to optical-frequency-comb generation with
  monolithic whispering-gallery-mode resonators.
\newblock {\em Physical Review A}, 82:033801, 2010.

\bibitem{Chembo_Menyuk}
Yanne~K. Chembo and Curtis~R. Menyuk.
\newblock Spatiotemporal {L}ugiato-{L}efever formalism for {K}err-comb
  generation in whispering-gallery-mode resonators.
\newblock {\em Phys. Rev. A}, 87:053852, 2013.

\bibitem{CrRab_bifurcation}
Michael~G. Crandall and Paul~H. Rabinowitz.
\newblock Bifurcation from simple eigenvalues.
\newblock {\em J. Functional Analysis}, 8:321--340, 1971.

\bibitem{DelHara_periodic}
Lucie Delcey and Mariana Haragus.
\newblock Periodic waves of the {L}ugiato-{L}efever equation at the onset of
  {T}uring instability.
\newblock {\em Philos. Trans. Roy. Soc. A}, 376(2117):20170188, 21, 2018.

\bibitem{Diddams1999}
S~A Diddams, Th~Udem, J~C Bergquist, E~A Curtis, R~E Drullinger, L~Hollberg,
  W~M Itano, W~D Lee, C~W Oates, K~R Vogel, D~J Wineland, J~Reichert, and
  R~Holzwarth.
\newblock {An Optical Clock Based on a Single Trapped 199 Hg+ Ion}.
\newblock {\em Science (New York, N.Y.)}, 24(13):881--883, 1999.

\bibitem{Godey_BifurcationAnalysis}
Cyril Godey.
\newblock A bifurcation analysis for the {L}ugiato-{L}efever equation.
\newblock {\em The European Physical Journal D}, 71(5):131, May 2017.

\bibitem{Godey_et_al2014}
Cyril Godey, Irina~V. Balakireva, Aur\'elien Coillet, and Yanne~K. Chembo.
\newblock Stability analysis of the spatiotemporal {L}ugiato-{L}efever model
  for {K}err optical frequency combs in the anomalous and normal dispersion
  regimes.
\newblock {\em Phys. Rev. A}, 89:063814, 2014.

\bibitem{Hansson2016}
Tobias Hansson and Stefan Wabnitz.
\newblock {Dynamics of microresonator frequency comb generation: Models and
  stability}.
\newblock {\em Nanophotonics}, 5(2):231--243, 2016.

\bibitem{Herr2013}
T.~Herr, V.~Brasch, J.~Jost, C.Y. Wang, N.M. Kondratiev, M.L. Gorodetsky, and
  T.J. Kippenberg.
\newblock Temporal solitons in optical microresonators.
\newblock {\em Nature Photonics}, 8:145--152, 2014.

\bibitem{herr_2012}
T.~Herr, K.~Hartinger, J.~Riemensberger, C.Y. Wang, E.~Gavartin, R.~Holzwarth,
  M.L. Gorodetsky, and T.J. Kippenberg.
\newblock Universal formation dynamics and noise of {K}err-frequency combs in
  microresonators.
\newblock {\em Nature Photonics}, 6:480--487, 2012.

\bibitem{jami:14}
Tobias Jahnke, Marcel Mikl, and Roland Schnaubelt.
\newblock Strang splitting for a semilinear {S}chr\"{o}dinger equation with
  damping and forcing.
\newblock {\em J. Math. Anal. Appl.}, 455(2):1051--1071, 2017.

\bibitem{Kielh_bifurcation_theory}
Hansj{\"o}rg Kielh{\"o}fer.
\newblock {\em Bifurcation theory}, volume 156 of {\em Applied Mathematical
  Sciences}.
\newblock Springer, New York, second edition, 2012.
\newblock An introduction with applications to partial differential equations.

\bibitem{Lau:15}
Ryan K.~W. Lau, Michael R.~E. Lamont, Yoshitomo Okawachi, and Alexander~L.
  Gaeta.
\newblock Effects of multiphoton absorption on parametric comb generation in
  silicon microresonators.
\newblock {\em Opt. Lett.}, 40(12):2778--2781, Jun 2015.

\bibitem{Lugiato_Lefever1987}
L.~A. Lugiato and R.~Lefever.
\newblock Spatial dissipative structures in passive optical systems.
\newblock {\em Phys. Rev. Lett.}, 58:2209--2211, 1987.

\bibitem{Mandel_secondary}
Rainer Mandel.
\newblock Global secondary bifurcation, symmetry breaking and period-doubling.
\newblock arXiv:1803.04903, 2018.
\newblock To appear in Topological Methods of Nonlinear Analysis.

\bibitem{MaRe_aprioribounds}
Rainer Mandel and Wolfgang Reichel.
\newblock A priori bounds and global bifurcation results for frequency combs
  modeled by the {L}ugiato-{L}efever equation.
\newblock {\em SIAM J. Appl. Math.}, 77(1):315--345, 2017.

\bibitem{Marin-Palomo2016}
Pablo Marin-Palomo, Juned~N. Kemal, Maxim Karpov, Arne Kordts, Joerg Pfeifle,
  Martin~H.P. Pfeiffer, Philipp Trocha, Stefan Wolf, Victor Brasch, Miles~H.
  Anderson, Ralf Rosenberger, Kovendhan Vijayan, Wolfgang Freude, Tobias~J.
  Kippenberg, and Christian Koos.
\newblock {Microresonator-based solitons for massively parallel coherent
  optical communications}.
\newblock {\em Nature}, 546(7657):274--279, 2017.

\bibitem{Miyaji_Ohnishi_Tsutsumi2011_Erratum}
T.~Miyaji, I.~Ohnishi, and Y.~Tsutsumi.
\newblock Erratum: Stability of a stationary solution for the
  {L}ugiato-{L}efever equation.
\newblock To appear in Tohoku Math. J.

\bibitem{Miyaji_Ohnishi_Tsutsumi2010}
T.~Miyaji, I.~Ohnishi, and Y.~Tsutsumi.
\newblock Bifurcation analysis to the {L}ugiato-{L}efever equation in one space
  dimension.
\newblock {\em Phys. D}, 239(23-24):2066--2083, 2010.

\bibitem{Miyaji_Ohnishi_Tsutsumi2011}
T.~Miyaji, I.~Ohnishi, and Y.~Tsutsumi.
\newblock Stability of a stationary solution for the {L}ugiato-{L}efever
  equation.
\newblock {\em Tohoku Math. J. (2)}, 63(4):651--663, 2011.

\bibitem{MiyTsu_existence}
Tomoyuki Miyaji and Yoshio Tsutsumi.
\newblock Existence of global solutions and global attractor for the third
  order {L}ugiato-{L}efever equation on {$\bold{T}$}.
\newblock {\em Ann. Inst. H. Poincar\'{e} Anal. Non Lin\'{e}aire},
  34(7):1707--1725, 2017.

\bibitem{MiyTsu_Steadystate}
Tomoyuki Miyaji and Yoshio Tsutsumi.
\newblock Steady-state mode interactions of radially symmetric modes for the
  {L}ugiato-{L}efever equation on a disk.
\newblock {\em Commun. Pure Appl. Anal.}, 17(4):1633--1650, 2018.

\bibitem{ParRivGomKnob_Bifurcation}
P.~Parra-Rivas, D.~Gomila, L.~Gelens, and E.~Knobloch.
\newblock Bifurcation structure of localized states in the {L}ugiato-{L}efever
  equation with anomalous dispersion.
\newblock {\em Phys. Rev. E}, 97(4):042204, 20, 2018.

\bibitem{Parra-Rivas2014}
P.~Parra-Rivas, D.~Gomila, M.~A. Mat{i}as, S.~Coen, and L.~Gelens.
\newblock Dynamics of localized and patterned structures in the lugiato-lefever
  equation determine the stability and shape of optical frequency combs.
\newblock {\em Phys. Rev. A}, 89:043813, Apr 2014.

\bibitem{Parra-Rivas2016}
P.~Parra-Rivas, E.~Knobloch, D.~Gomila, and L.~Gelens.
\newblock {Dark solitons in the Lugiato-Lefever equation with normal
  dispersion}.
\newblock {\em Physical Review A}, 93(6):1--17, 2016.

\bibitem{Pfeifle_Koos3}
J.~Pfeifle, V.~Brasch, M.~Lauermann, Y.~Yu, D.~Wegner, T.~Herr, K.~Hartinger,
  P.~Schindler, J.~Li, D.~Hillerkuss, R.~Schmogrow, C.~Weimann, R.~Holzwarth,
  W.~Freude, J.~Leuthold, T.~J. Kippenberg, and C.~Koos.
\newblock Coherent terabit communications with microresonator {K}err frequency
  combs.
\newblock {\em Nature Photon.}, 8:375--380, 2014.

\bibitem{Pfeifle2015}
Joerg Pfeifle, Aur{\'{e}}lien Coillet, R{\'{e}}mi Henriet, Khaldoun Saleh,
  Philipp Schindler, Claudius Weimann, Wolfgang Freude, Irina~V. Balakireva,
  Laurent Larger, Christian Koos, and Yanne~K. Chembo.
\newblock {Optimally coherent Kerr combs generated with crystalline whispering
  gallery mode resonators for ultrahigh capacity fiber communications}.
\newblock {\em Physical Review Letters}, 114(9):1--5, 2015.

\bibitem{Rab_some_global}
Paul~H. Rabinowitz.
\newblock Some global results for nonlinear eigenvalue problems.
\newblock {\em J. Functional Analysis}, 7:487--513, 1971.

\bibitem{StanStef_asymptotic}
Milena Stanislavova and Atanas~G. Stefanov.
\newblock Asymptotic stability for spectrally stable {L}ugiato-{L}efever
  solitons in periodic waveguides.
\newblock {\em J. Math. Phys.}, 59(10):101502, 12, 2018.

\bibitem{Suh2017}
Myoung-Gyun Suh and Kerry Vahala.
\newblock {Soliton Microcomb Range Measurement}.
\newblock {\em Science}, 887(February):884--887, 2017.

\bibitem{Trocha2017}
P.~Trocha, M.~Karpov, D.~Ganin, M.~H.P. Pfeiffer, A.~Kordts, S.~Wolf,
  J.~Krockenberger, P.~Marin-Palomo, C.~Weimann, S.~Randel, W.~Freude, T.~J.
  Kippenberg, and C.~Koos.
\newblock {Ultrafast optical ranging using microresonator soliton frequency
  combs}.
\newblock {\em Science}, 359(6378):887--891, 2018.

\bibitem{Udem2002}
Th. Udem, R.~Holzwarth, and T.~W. H{\"{a}}nsch.
\newblock {Optical frequency metrology}.
\newblock {\em Nature}, 416(6877):233--237, 2002.

\end{thebibliography}

\section*{Acknowledgments}

We gratefully acknowledge financial support by the Deutsche Forschungs\-gemeinschaft (DFG) through CRC 1173.

\end{document}